\documentclass{aims}
\usepackage{amsfonts}
\usepackage{geometry}
\usepackage{amsmath}
\usepackage{amssymb}
\usepackage{graphicx}
\usepackage{txfonts}
\def\typeofarticle{Research article}
\def\currentvolume{3}
\def\currentissue{1}
\def\currentyear{2018}

\def\ppages{148--182}
\def\DOI{10.3934/Math.2018.1.148}
\def\Received{19 February 2018}
\def\Accepted{11 March 2018}
\def\Published{16 March 2018}
\newtheorem{theorem}{Theorem}[section]
\newtheorem{lemma}[theorem]{Lemma}
\newtheorem{definition}[theorem]{Definition}
\newtheorem{example}[theorem]{Example}

\newtheorem{remark}[theorem]{Remark}
\numberwithin{equation}{section}

\begin{document}
\newtheorem{pf}{Proof}
\newtheorem{pot}{Proof of Theorem}
\title{A minimization approach to conservation laws with random initial conditions
and non-smooth, non-strictly convex flux}
\author{Carey Caginalp\affil{}\corrauth}
\shortauthors{the Author(s)}
\address{University of Pittsburgh, Mathematics Department, 301 Thackeray Hall,
Pittsburgh PA\ 15260}
\corraddr{Email: carey\_caginalp@alumni.brown.edu.}

\begin{abstract}
We obtain solutions to conservation laws under any random initial conditions
that are described by Gaussian stochastic processes (in some cases discretized). We analyze the
generalization of Burgers' equation for a smooth flux function $H\left(  p\right)
=\left\vert p\right\vert ^{j}$ for $j\geq2$ under random
initial data. We then consider a piecewise linear, non-smooth and non-convex
flux function paired with general discretized Gaussian stochastic process initial data. By
partitioning the real line into a finite number of points, we obtain an exact
expression for the solution of this problem. From this we can also find exact
and approximate formulae for the density of shocks in the solution profile at
a given time $t$ and spatial coordinate $x$. We discuss the simplification of
these results in specific cases, including Brownian motion and Brownian
bridge, for which the inverse covariance matrix and corresponding eigenvalue
spectrum have some special properties. We calculate the transition 
probabilities between various cases and examine the variance of the 
solution $w\left(x,t\right)$ in both $x$ and $t$. We also describe how results may be obtained for a non-discretized version of a Gaussian stochastic process by
taking the continuum limit as the partition becomes more fine.

\end{abstract}
\maketitle

\section{Introduction}

The conservation law in the form%
\begin{align}
w_{t}+\left(  \mathcal{H}\left(  w\right)  \right)  _{x} &  =0\text{ in
}\mathbb{R\times}\text{ }\left(  0,\infty\right)  \nonumber\\
w\left(  x,0\right)   &  =g^{\prime}\left(  x\right)  \text{ on }%
\mathbb{R}\times\left\{  t=0\right\}  \label{cl}%
\end{align}
and its related Hamilton-Jacobi problem%
\begin{align}
u_{t}+\mathcal{H}\left(  u_{x}\right)   &  =0\text{ in }\mathbb{R\times}\text{
}\left(  0,\infty\right)  \nonumber\\
u\left(  x,0\right)   &  =g\left(  x\right)  \text{ on }\mathbb{R}%
\times\left\{  t=0\right\}  \label{hj}%
\end{align}
have a wide array of applications among fluid mechanics, shocks, and
turbulence. An interesting feature is that even for a smooth flux function
$\mathcal{H}$ and smooth initial data, discontinuous solutions to the
conservation law (\ref{cl}) may occur \cite{BG, CD, ERS, EV, FM, GR, HA, HP,
KR, KA, LA1, LA2, M, MP, MS, SL1, SH}. In particular, this is observed in the
prototypical case of Burgers' equation, for which $\mathcal{H}\left(
w\right)  =w^{2}/2$. The work of Dafermos \cite{D1} was instrumental in
constructing an approximation for this case and more generally, locally
Lipschitz $\mathcal{H}$ by using a series of piecewise linear flux functions
to construct the solution for deterministic initial conditions. Using our
methods for random initial conditions, one also can construct the continuum
limit for the deterministic problem as a special case.

In this paper, we consider the introduction of randomness into the initial
conditions for the conservation law above. The classical approach for solving
the conservation law (\ref{cl}) entails using variational methods to derive
the Lax-Oleinik formula \cite{EV, HP, LA1, LA2}%
\begin{equation}
w\left(  x,t\right)  :=\frac{\partial}{\partial x}\left[  \min_{y\in
\mathbb{R}}\left\{  t\mathcal{L}\left(  \frac{x-y}{t}\right)  +g\left(
y\right)  \right\}  \right]  , \label{clsol}%
\end{equation}
requiring assumptions on smoothness, superlinearity, and strict convexity of
the flux function. In a previous paper \cite{CA}, we showed that (\ref{clsol})
will still hold when these conditions are relaxed, with the only other
assumption being that of Lipschitz continuity on the initial condition $g^{\prime
}\left(  x\right)  $. These methods thus enable us to derive a number of
results for the case of non-deterministic initial data. When dealing with
randomness, an important notion is that of a stochastic process $X\left(
t\right)  $. The process is a set of random variables linked together by a law
of probability with the set of possible paths given by some state space
$\Omega$. In general this is far too broad a class to attack analytically, so
attempts to derive exact solutions are often restricted to one specific
stochastic process, for example Brownian motion \cite{FM}. In this work, we
derive results for an entire family of stochastic processes. In Section 2 we
provide an introduction and references to some of the key properties of such a
family known as \textit{Gaussian stochastic processes}. The motivation for
this class of process is not only that it has sufficient structural properties
to obtain closed-form results but also has the breadth to analyze a number of
distinct cases. We also utilize this approach in calculating transition 
probabilities between different states and examining the variance in 
space as time increases.

Consider the basic problem $w_{t} + |w|_{x} = 0$ subject to continuous initial
conditions $w(x,0)=g'(x)$. Application of our results gives 
$g(x-t)$ as the solution if the minimum of $g(y)$ is at the left endpoint of the interval, 
$g(x+t)$ if the minimum is on the right endpoint, and $0$ for an 
interior minimum. Thus for $g'$ that is given by Brownian motion, 
one minimizes integrated Brownian motion over a finite interval and
obtains the shock statistics.

This example can provide a pathway to considering the behavior 
of solutions to more complicated flux functions under random initial data, 
for example when one has not one but a number of affine
segments linked together. This sort of construction is denoted as a 
\emph{polygonal flux function} and is thus particularly suited to analysis
under the Hopf-Lax minimization approach \cite{HP, LA1,
LA2} discussed below. Together with the results proven by Dafermos \cite{D1},
this methodology can be the basis for a general class of flux and random
initial data.

Notably, both Brownian motion and Brownian bridge initial conditions for
$g^{\prime}\left(  x\right)  $ will satisfy the requisite regularity
properties, as well the additional property of being Gaussian stochastic
processes. Furthermore, when integrated to obtain integrated Brownian motion
and integrated Brownian bridge, they remain Gaussian stochastic process,
leading to useful properties on both $g$ and $g^{\prime}$ that aid in the
analysis and computation of minimizers as specified in (\ref{clsol}). As is
shown in \cite{CA}, we have the key relationship for the solution $w$ given by%
\begin{equation}
w\left(  x,t\right)  =g^{\prime}\left(  y^{\ast}\left(  x,t\right)  \right)
\label{introkr}%
\end{equation}
(when $g'\left(y^{\ast}\left(x,t\right)\right)$ exists) 
expressed as a function of the greatest minimizer,%
\begin{equation}
y^{\ast}=\arg^{+}\min_{y\in\mathbb{R}}\left\{  t\mathcal{L}\left(  \frac
{x-y}{t}\right)  +g\left(  y\right)  \right\}  . \label{ystardef}%
\end{equation}
For the case of Burgers' equation, one uses the relation $g^{\prime}\left(
y^{\ast}\left(  x,t\right)  \right)  =\frac{x-y^{\ast}\left(  x,t\right)  }%
{t}$ to obtain a special case of (\ref{introkr}) given by%
\begin{equation}
w\left(  x,t\right)  =\frac{x-y^{\ast}\left(  x,t\right)  }{t}
\label{introkrb}%
\end{equation}
When one introduces Brownian motion initial data, (\ref{introkr}) leads to the
following result for the variance of the solution expressed in terms of the
variance of the minimizer in (\ref{ystardef}), for a flux function given by
$\mathcal{H}\left(  p\right)  =\frac{1}{j}\left\vert p\right\vert ^{j}$%
\begin{equation}
Var\left(  w\left(  x,t\right)  \right)  =t^{-\frac{2}{j-1}}Var\left(
sgn\left(  x-y^{\ast}\left(  x,t\right)  \right)  \left\vert x-y^{\ast}\left(
x,t\right)  \right\vert ^{\frac{1}{j-1}}\right)
\end{equation}
which for Burgers' equation reduces to
\begin{equation}
Var\left(  w\left(  x,t\right)  \right)  =\frac{1}{t^{2}}Var\left(  x-y^{\ast
}\left(  x,t\right)  \right)  . \label{varburg}
\end{equation}
Later on, we specialize to the case $x=0$.

In Section 3 we discuss extensions of (\ref{varburg}) to more general flux
functions. For example, with the flux $\mathcal{H}\left(  p\right)  =p^{2n}$,
we have a generalization of (\ref{varburg}), providing an understanding of the
propagation of variance over the sample space $\Omega$. We also consider a
related case of $\mathcal{H}\left(  p\right)  =\left\vert p\right\vert $, in
light of its Legendre transform%
\begin{equation}
\mathcal{L}\left(  q\right)  =\left\{
\begin{array}
[c]{c}%
0\\
+\infty
\end{array}
\right.
\begin{array}
[c]{c}%
\left\vert q\right\vert \leq1\\
\left\vert q\right\vert >1
\end{array}
\end{equation}
for which the solution has a very simple probabilistic interpretation. In
particular, one can show formally that the frequency of jumps in $w\left(  x,t\right)
$ decreases in time.

The minimum defined on the right-hand side of (\ref{ystardef}) is of course
unique under the classical assumptions that $\mathcal{H}$ is smooth,
superlinear, and uniformly convex \cite{EV}. However, the analog of the result
(\ref{introkr}) remarkably continues to hold even when these assumptions are
relaxed, i.e. when the flux function has a lesser degree of regularity, the
convexity of $\mathcal{H}$ is not uniform or even strict, and when it is no
longer superlinear. Even if the initial condition $g^{\prime}$ is also
discontinuous, one can prove a modified form of (\ref{introkr}) by accounting
for the vertices of $g$.

A particularly interesting and yet broad range of cases results from studying
a Gaussian stochastic process, $X\left(  r\right)  $, a random process
completely determined by its second order statistics. In other words, such a
process at arbitrary times $\left\{  r_{i}\right\}  _{i=1}^{n}$ is completely
determined by the means $\mathbb{E}\left\{  X\left(  r_{i}\right)  \right\}  $
and the corresponding covariance matrix $\sum$ defined by $\left(
\sum\right)  _{ij}=Cov\left(  X\left(  r_{i}\right)  ,X\left(  r_{j}\right)
\right)  $. We denote the inverse covariance matrix by $A:=\sum^{-1}$.

In Section 4, we consider a Gaussian stochastic process $X\left(  r\right)  $
and partition the real line into sets $\zeta_{i}$ with measure $\nu\left(
\zeta_{i}\right)  >\varepsilon>0$, and take a finite number of points, among
which we compute the minimum. We pair this initial condition with a piecewise
linear flux function. The Legendre transform $L$ of this flux function only
takes a finite number of slopes, each on an interval of finite measure. By
taking the intersection of our partition with this interval, we obtain a
finite covering for it, and thus a finite number of points at which to analyze
the quantity minimized in (\ref{ystardef}). We want to calculate the solution
$\mathbb{E}\left\{  w\left(  x,t\right)  \right\}  $, which we have shown in
\cite{CA} to depend only on the derivative of $L$ or $g$ at the point of the
minimizer. By considering first the local minima of this quantity for each of
the finite number of segments of $L$, and then taking the minimum over those
segments, we obtain a remarkable closed-form expression for the solution. 
Among the results we can obtain from this is the expectation of the solution 
at a point $\left(x,t\right)$.
\begin{align}
\mathbb{E}\left\{  w\left(  x,t\right)  \right\}
=-\sum_{i=1}^{N}p_{i}c_{N+1-i}
+ \sum_{i=N+1}^{2N+1}p_{i}\mathbb{E}\left\{g'(y^{*}(x,t))\right\}\label{Thm5}%
\end{align}
where%
\begin{align}
p_{i} &  = \sum_{j=2}^{n_{i}}
\left(  2\pi\right)  ^{-\frac{n}{2}}\left\vert A\right\vert
^{\frac{1}{2}}\int_{-\infty}^{\infty}dx_{i,j}
{\displaystyle\prod\limits_{\left(m,l\right)\neq \left(i,j\right)}}\int_{x_{i,j}}^{\infty} dx_{m,l}
e^{-\frac{1}{2}\left(
\bar{x}-\tilde{\mu}\right)  ^{T}A\left(  \bar{x}-\tilde{\mu}\right)}\nonumber\\
&for \ 1\le i \le N.\nonumber\\
&\ For \ N+1\leq i\le2N+1 \ one \ has \ p_{i}=\mathbb{P}\left\{R_{i-N}\right\}.
\label{introSolution}
\end{align}
are expressed in terms of integrals of nested error functions, and the sets 
$\left\{R_{i}\right\}$ refer to the event that a minimum occurs at a vertex
of $L$ (discussed further in (\ref{Thm4})). The quantity
$A_{i}$ in the first term of (\ref{introSolution}) here denotes the inverse
covariance matrix corresponding to the $i$th affine segment of $L$, and $A$
for the same quantity over the whole domain. The vector $\mu$ is given by the
discrete points of $\mu=\mathbb{E}\left\{  g\left(  y\right)  +tL\left(
\frac{x-y}{t}\right)  \right\}  $, where the second term is deterministic. The
quantity $\tilde{\mu}$ refers to the analogous quantity in the $i$th segment.
The integers $n_{i}$ correspond to technical details of partitioning the
intervals for which $L$ has different slopes, as described in Section 5.

An important aspect of this evolution is that for Burgers' equation with random
initial conditions, one has an increasing variance in the probabilistic sense
for the solution $w\left(  x,t\right)  $ as a function of $t$ for a given $x$,
but formal calculations suggest that it has decreasing variance in the sense
of $TV\left(  w\left(  x,t\right)  \right)  $ as a function of time. In this
manner, it is almost paradoxical that conservation laws which lead to shocks
from smooth solutions can also act in such a way as to smooth out random
initial conditions.

A key issue in kinetic theory is the distribution and formation of
discontinuities in the solution, known as \textit{shocks}. In particular, for
a Gaussian stochastic process paired with a piecewise linear flux function, we
can expand on the results of Section 5, including (\ref{introSolution}). In
doing so, we derive expressions (both exact and approximate) for the
distribution of these shocks based on the minimizer switching from a point on
one segment of $L$ to another. From this result, we also compute the total
density of shocks.

The expressions obtained in Sections 5 and 6 hold for \textit{any} Gaussian
stochastic process. Some cases of particular interest include that of Brownian
motion or Brownian bridge initial conditions, as well as Ornstein-Uhlenbeck
(\cite{PK}, p. 439). In Section 6, we apply our results to special cases and
observe the simplifications for these specific processes. Since many of our
results are exact and in terms of error functions, this also provides a
pathway to further numerical computation. By taking the limit of a finer
partition of the real line, we also explain how one may be able to pass to a
limit and obtain results for the continuum problem in addition to the
discretized one.

\section{Background on Gaussian stochastic processes}

\subsection{Stochastic processes}

In considering randomness for the initial conditions along the entire real
line $\mathbb{R}$, it is important to have a well-defined construction. A
\textit{stochastic process} $X\left(  s\right)  $ is a collection of random
variables on a probability space $\left(  \Omega,\mathcal{F},\mathbb{P}%
\right)  $. The space $\Omega$ can be thought of as the set of all possible
outcomes for the random variable, $\mathcal{F}$ as a set of reasonable sets of
outcomes (i.e. \textit{measurable} sets) and $\mathbb{P}$ is a measure with
total mass $1$ assigning the probabilities for which an outcome or set of
outcomes in $\mathcal{F}$ can occur. The expectation $\mathbb{E}\left\{
X\right\}  $ of a random variable $X$ is defined as $\mathbb{E}\left\{
X\right\}  =\int_{\Omega}X\left(  \omega\right)  dP\left(  \omega\right)  $.
We also require that the measure $\mathbb{P}$ be countably additive, i.e.
$\mathbb{P}\left(  \cup_{i\in\mathbb{N}}F_{i}\right)  =\sum_{i=1}^{\infty
}\mathbb{P}\left(  F_{i}\right)  $ for $\left\{  F_{i}\right\}  $ disjoint.

\begin{example}
\label{Ex 2.2}Consider a collection of random variables $\left\{  X_{\alpha
}\right\}  _{\alpha\in I}$ for some index set $I$ such that $X_{\alpha}%
\sim\mathcal{N}\left(  0,1\right)  $. That is, take identically distributed
random variables that are normally distributed with mean $0$ and variance $1$.
Define%
\begin{equation}
S\left(  n\right)  =\sum_{i=1}^{n}X_{\alpha}.
\end{equation}
The process $S\left(  n\right)  $ is known as a random walk and is a
stochastic process on the probability space $\left(  \Omega,\mathcal{F}%
,\mathbb{P}\right)  $ where $\Omega=\mathbb{N}$, $\mathcal{F}$ is the set of
Borel sets on $\mathbb{R}$, and $\mathbb{P}$ is the usual probability measure.
\end{example}

\subsection{Gaussian stochastic processes}

Let $\left\{  R_{i}\right\}  _{i=1}^{m}$ be a collection of independent,
identically distributed random variables with standard normal distribution,
i.e. $R_{i}\sim\mathcal{N}\left(  0,1\right)  $ as in Example \ref{Ex 2.2}. If
a set of random variables $\left\{  X_{i}\right\}  _{i=1}^{n}$ satisfy%
\begin{align}
X_{1} &  =a_{11}R_{1}+...+a_{1m}R_{m}+\mu_{1}\nonumber\\
&  ...\nonumber\\
X_{n} &  =a_{n1}R_{1}+...+a_{nm}R_{m}+\mu_{n}%
\end{align}
for constants $\left\{  a_{ij}\right\}  ,$ $\left\{  \mu_{i}\right\}  $, we
say that $\left(  X_{1},...,X_{n}\right)  $ form a \textit{multivariate normal
distribution}. We also define the \textit{covariance } of two random variables below.

\begin{definition}
Let $X_{i}$ and $X_{j}$ be random variables. We define the \textit{covariance
}of the random variables by%
\begin{align}
Cov\left[  X_{i},X_{j}\right]   &  =\mathbb{E}\left[  X_{i}-\mathbb{E}\left[
X_{i}\right]  \right]  \mathbb{E}\left[  X_{j}-\mathbb{E}\left[  X_{j}\right]
\right] \nonumber\\
&  =\mathbb{E}\left[  X_{i}X_{j}\right]  -\mathbb{E}\left[  X_{i}\right]
\mathbb{E}\left[  X_{j}\right]
\end{align}

\end{definition}

This is the basis of the definition for a Gaussian stochastic process. In
particular, one has the following \cite{RO}:

\begin{theorem}
If $\left(  X_{1},...,X_{n}\right)  $ have a multivariate normal distribution with 
density given by
\begin{equation}
f_{X}\left(X_{1},...,X_{n}\right)=\frac{e^{-\frac{1}{2}\left(\overrightarrow{x}-\overrightarrow{\mu}\right)^{T}
\Sigma^{-1}\left(\overrightarrow{x}-\overrightarrow{\mu}\right)}}
{\sqrt{\left(2\pi\right)^{n}\left\vert\Sigma\right\vert}}
\end{equation}
where $\overrightarrow{\mu}$ is the mean, and $\Sigma$ the covariance 
matrix, then the joint moment generating function is given by%
\begin{equation}
\phi\left(  \lambda_{1},...,\lambda_{n}\right)  =\exp\left\{  \sum_{i=1}%
^{n}\lambda_{i}\mu_{i}+\frac{1}{2}\sum_{i=1}^{n}\sum_{j=1}^{n}\lambda
_{i}\lambda_{j}Cov\left[  X_{i}X_{j}\right]  \right\}
\end{equation}
\end{theorem}

This is easily proven using the definition and algebraic manipulation and
properties of characteristic functions. We can now define what we mean by a
Gaussian stochastic process:

\begin{definition}
\label{Def 2.1}Let $\left\{  t_{i}\right\}  _{i=1}^{n}$ be arbitrary and
increasing without loss of generality. We call a stochastic process $\left\{
X\left(  t\right)  \right\}  $ a \textit{Gaussian stochastic process if
}$\left(  X\left(  t_{1}\right)  ,...,X\left(  t_{n}\right)  \right)  $ has a
multivariate normal distribution.
\end{definition}

Definition \ref{Def 2.1} implies that a stochastic process that is Gaussian is
completely characterized by the properties of its mean and covariance 
matrix. A natural question to ask is what sort of processes might satisfy this 
requirement, and how one proves this property. We next give several 
examples relating to cases of particular interest in the broader literature. 
More details can be found in \cite{RO}.

\begin{definition}
\label{Def 2.2}Let $W\left(  t\right)  $ be a stochastic process. We call
$W\left(  t\right)  $ Brownian motion if

(i) $W\left(  0\right)  =0$

(ii) $W\left(  t\right)  $ is almost surely continuous

(iii)\ $W\left(  t\right)  $ has independent increments with a normal
distribution, i.e.%
\begin{equation}
W\left(  t\right)  -W\left(  s\right)  \sim\mathcal{N}\left(  0,t-s\right)  .
\end{equation}

\end{definition}

\begin{definition}
\label{Def 2.3}We call a Brownian motion conditioned on $W\left(  T\right)
=0$, denoted $\left\{  W_{b}\left(  t\right)  \right\}  _{0\leq t\leq T}$ a
\textit{Brownian bridge}. An equivalent definition is to set%
\begin{equation}
W_{b}\left(  t\right)  =W\left(  t\right)  -\frac{t}{T}W\left(  T\right)
\text{.}%
\end{equation}
In particular, we note that $W_{b}\left(  0\right)  =W_{b}\left(  T\right)
=0$.
\end{definition}

\begin{definition}
We call a process $W_{o}\left(  t\right)  $ stationary Ornstein-Uhlenbeck if
it is defined by%
\begin{equation}
W_{o}\left(  t\right)  =e^{-\alpha t}W\left(  e^{2\alpha t}t\right)  .
\end{equation}

\end{definition}

The stationary Ornstein-Uhlenbeck process has the important property that its
associated covariance matrix is invariant under translation in time. 
A basic result is the following.

\begin{lemma}
\label{Prop 2.3}(a)\ Let $\left\{  W\left(  t\right)  \right\}  _{t\geq0}$ be
a Brownian motion. Then $W\left(  t\right)  $ is a Gaussian stochastic process.

(b)\ The \textit{integrated Brownian motion} $S\left(  t\right)  $ given by
$S\left(  t\right)  :=\int_{0}^{t}W\left(  r\right)  dr$. Then $S\left(
t\right)  $ is a Gaussian stochastic process. Further, defining $Z\left(
t\right)  $ as integrated Brownian motion together with a drift, $Z\left(
t\right)  $ is a Gaussian stochastic process.

(c) Let $T>0$ be given and $W\left(  t\right)  $ be a Brownian motion on
$\left[  0,T\right]  $. The stochastic process $W_{b}\left(  t\right)
=W\left(  t\right)  -\frac{t}{T}W\left(  T\right)  $, called \textit{Brownian
bridge}, is a Gaussian stochastic process.

(d) Define $Z\left(  t\right)  =\int_{0}^{t}W_{b}\left(  r\right)  dr$. Then
$Z\left(  t\right)  $ is also a Gaussian stochastic process.
\end{lemma}

Part (b) is proven in the Appendix. A basic result that follows from the
standard theory is given below.

\begin{theorem}
\label{Thm 2.1}(\textbf{Covariance Matrices for Brownian Motion and Brownian
Bridge}). Let $T$ and $\left\{  t_{i}\right\}  _{i=1}^{N}\in\left[
0,T\right]  $ be given.

(a) Let $\left\{  S\left(  t\right)  \right\}  _{t>0}$ be an integrated
Brownian motion with piecewise linear drift. Then the covariance at times
$0\leq s\leq t$ is given by%
\begin{equation}
Cov\left[  S\left(  t\right)  ,S\left(  s\right)  \right]  =s^{2}\left(
\frac{t}{2}-\frac{s}{6}\right)  \label{covMatrixIBM}%
\end{equation}

(b) For an integrated Brownian bridge $Z\left(  t\right)  =\int_{0}%
^{t}W\left(  r\right)  -\frac{r}{T}W\left(  T\right)  dr$, the covariance
matrix at times $0\leq s\leq t\leq T$ is given by%
\begin{equation}
Cov\left[  Z\left(  t\right)  ,Z\left(  s\right)  \right]  =s^{2}\left(
\frac{t}{2}-\frac{s}{6}\right)  -\frac{1}{T}\frac{t^{2}s^{2}}{4}
\label{covMatrixIBB}%
\end{equation}
For $s$ and $t$ outside of $\left[  0,T\right]  $, the variance of $Z\left(
t\right)  $ is constant but nonzero.
\end{theorem}

An outline of the proof of Theorem \ref{Thm 2.1} is provided in the Appendix.
\begin{remark}
Later on, we consider Brownian motion on the interval $\left(0,1\right)$. 
In order to use our results on a different interval, one can instead consider 
Brownian motion on a translated and stretched interval, i.e. $\left(-M,M\right)$ 
where $B\left(-M\right)=0$ is the starting point instead of $0$ as is typically 
used.\label{BMextend}
\end{remark}

\section{Generalizations of Burgers' equation with random initial data}

We consider a generalization of Burgers' equation by considering the flux
function for $j\geq2$ as follows:%
\begin{equation}
\mathcal{H}\left(  p\right)  =\frac{p^{j}}{j}. \label{genfluxfn}%
\end{equation}
In this way, we obtain the following generalization of (\ref{varburg}).

\begin{lemma}
\label{Prop 2.1}With $\mathcal{H}$ defined by (\ref{genfluxfn}), consider the
conservation law%
\begin{align}
w_{t}+\left(  \mathcal{H}\left(  w\right)  \right)  _{x}  &  =0\text{ in
}\mathbb{R\times}\text{ }\left(  0,\infty\right) \nonumber\\
w\left(  x,0\right)   &  =g^{\prime}\left(  x\right)  \text{ on }%
\mathbb{R}\times\left\{  t=0\right\}.
\end{align}
The solution is then%
\begin{equation}
w\left(  x,t\right)  =sgn\left(  \frac{x-y^{\ast}\left(  x,t\right)  }%
{t}\right)  \left(  \frac{x-y^{\ast}\left(  x,t\right)  }{t}\right)
^{1/\left(  j-1\right)  }. \label{Propgenburgers}%
\end{equation}

\end{lemma}

\begin{proof}
The Legendre transform of this flux function (\ref{genfluxfn}) is%
\begin{align}
\mathcal{L}\left(  q\right)   &  =\sup_{p}\left\{  pq-p^{j}\right\}
=\left\vert q\right\vert ^{j/\left(  j-1\right)  }-\frac{1}{j}\left\vert
q\right\vert ^{j/\left(  j-1\right)  }\nonumber\\
&  =\frac{j-1}{j}\left\vert q\right\vert ^{j/\left(  j-1\right)  }%
\end{align}
and its derivative is given by%
\begin{equation}
\mathcal{L}^{\prime}\left(  q\right)  =sgn\left(  q\right)  q^{1/\left(
j-1\right)  }.
\end{equation}
The solution $w\left(  x,t\right)  $ to the conservation law (\ref{cl}) can
then be expressed as%
\[
w\left(  x,t\right)  =\mathcal{L}^{\prime}\left(  \frac{x-y^{\ast}\left(
x,t\right)  }{t}\right)  =sgn\left(  \frac{x-y^{\ast}\left(  x,t\right)  }%
{t}\right)  \left\vert \frac{x-y^{\ast}\left(  x,t\right)  }{t}\right\vert
^{1/\left(  j-1\right)  }%
\]

\end{proof}

\begin{remark}
For Burgers' equation, $j=2$ and the relation (\ref{Propgenburgers}) reduces
to%
\[
w\left(  x,t\right)  =\frac{x-y^{\ast}\left(  x,t\right)  }{t},
\]
i.e., the result (\ref{introkrb}).
\end{remark}

\begin{lemma}
\label{Prop 2.2}Consider the conservation law as in Lemma \ref{Prop 2.1} with
$\mathcal{H}\left(  p\right)  =\left\vert p\right\vert $ with initial data
$g^{\prime}$. Then the solution is given by%
\begin{align}
w\left(  x,t\right)   &  =g^{\prime}\left(  y^{\ast}\left(  x,t\right)
\right)  \text{ at points where }g^{\prime}\text{ is continuous,}\nonumber\\
w\left(  x,t\right)   &  =0\text{ elsewhere,}%
\end{align}
where $y^{\ast}$ is the value of the minimizer as before.
\end{lemma}

\begin{proof}
The Legendre transform of $\mathcal{H}$ is simply%
\begin{equation}
L\left(  q\right)  =\left\{
\begin{array}
[c]{c}%
0\\
+\infty
\end{array}
\right.
\begin{array}
[c]{c}%
\left\vert q\right\vert \leq1\\
\left\vert q\right\vert >1
\end{array}
.
\end{equation}
The result then follows immediately from the solution formula (\ref{clsol}).
\end{proof}

This result is particularly of interest as it shows, for the special case of
the flux function $\mathcal{H}\left(  p\right)  =\left\vert p\right\vert $,
that the solution of the conservation law reduces to finding the minimum of
the integral of the initial data, i.e. the minimum of $g$. For example, if
$g^{\prime}$ is given by a Brownian motion, the solution at given $x$ and $t$
is given by solving the problem of finding the minimum of an integrated
Brownian motion on a finite interval (whose endpoints are dependent on $x$ and
$t$). The integrated Brownian motion is a Gaussian stochastic process (Lemma
\ref{Prop 2.3} (b)), so the minimizer $y^{\ast}$ can easily be approximated by
using the mean and covariance matrix (Theorem \ref{Thm 2.1} (a)).

For example, if the initial condition $g^{\prime}$ consists of $\pm1$ with
probabilities $p$ and $1-p$, then one has
$w\left(  x,t\right)  \in\left\{  0,g^{\prime}\left(  x-t\right)  ,g^{\prime
}\left(  x+t\right)  \right\}  $. The value taken by $w$ depends on whether
the minimum is attained on the interior, left endpoint, or right endpoint of
the interval $\left[  x-t,x+t\right]  $.

\begin{remark}
\label{Rmk Range}Due to the Legendre transform $L$ taking on infinite value
outside of $\left\vert q\right\vert >1$, for a given $\left(  x,t\right)  $,
the minimizer $y^{\ast}\left(  x,t\right)  $ takes its value in the domain%
\begin{equation}
\left\vert x-y\right\vert <t.
\end{equation}
This is a manifestation of the finite domain of influence strictly enforced,
i.e. within a cone $y=x\pm t$.
\end{remark}

\begin{theorem}
\label{Thm 2.2}Consider the conservation law (\ref{cl}) with flux function
given by (\ref{genfluxfn}), i.e. $\mathcal{H}\left(  p\right)  =\frac{p^{j}%
}{j}$, $j\geq2$. For Brownian motion initial conditions one has, with $y^{\ast}$
depending on $\left(  x,t\right)  $,\newline%
\begin{equation}
Var\left(  w\left(  x,t\right)  \right)  =t^{-\frac{2}{j-1}}Var\left(
sgn\left(  x-y^{\ast}\right)  \left\vert x-y^{\ast}\right\vert ^{\frac{1}%
{j-1}}\right),  \label{Sec3me}%
\end{equation}
noting that $y^{*}$ is short for $y^{*}\left(x,t\right)$.
\end{theorem}

\begin{proof}
For a flux function of the form (\ref{genfluxfn}), a minimum is attained when%
\begin{equation}
\mathcal{L}^{\prime}\left(  \frac{x-y^{\ast}}{t}\right)  =sgn\left(
\frac{x-y^{\ast}}{t}\right)  \left(  \frac{x-y^{\ast}}{t}\right)  ^{\frac
{1}{j-1}}=g^{\prime}\left(  y^{\ast}\right)  =w\left(  x,t\right)  .
\end{equation}
One then writes%
\begin{align}
\mathbb{E}\left\{  w\left(  x,t\right)  \right\}   &  =t^{-\frac{1}{j-1}%
}\mathbb{E}\left\{  sgn\left(  \frac{x-y^{\ast}}{t}\right)  \left\vert
x-y^{\ast}\right\vert ^{\frac{1}{j-1}}\right\}  ,\nonumber\\
\mathbb{E}\left\{  w\left(  x,t\right)  ^{2}\right\}   &  =t^{-\frac{2}{j-1}%
}\mathbb{E}\left\{  \left\vert x-y^{\ast}\right\vert ^{\frac{2}{j-1}}\right\}
.
\end{align}
Therefore%
\begin{align}
Var\left(  w\left(  x,t\right)  \right)    & =t^{-\frac{2}{j-1}}%
\{\mathbb{E}\left\{  \left\vert x-y^{\ast}\right\vert ^{\frac{2}{j-1}%
}\right\}  \\
& -\mathbb{E}\left\{  sgn\left(  \frac{x-y^{\ast}}{t}\right)  \left\vert
x-y^{\ast}\right\vert ^{\frac{1}{j-1}}\right\}  ^{2}\}
\end{align}

\end{proof}

\begin{remark}
To obtain the result for Burgers' equation, one sets $j=2$, and observes
that Theorem \ref{Thm 2.2} then yields the result (\ref{varburg}):%
\begin{equation}
Var\left(  w\left(  x,t\right)  \right)  =\frac{1}{t^{2}}Var\left( x- y^{\ast
}\left(  x,t\right)  \right)  .
\end{equation}
Further, for the case of Brownian motion and for an arbitrary time $t$,
$Var\left(g'\left(  y^{\ast}\left(  x,t\right)\right)  \right)  $ is increasing in $x$ 
since $y^{\ast}\left(x,t\right)$ is increasing in $x$ (see 
\cite{CA}), and
thus the variance of $w\left(  x,t\right)$ also increases.
\end{remark}

\begin{remark}
From (\ref{Sec3me}) for Burgers' equation at $x=0$, using Theorem 
\ref{Thm 2.2} we have
\begin{equation}
Var\left(  w\left( 0,t\right)  \right)  \text{ }\sim\text{ }\frac{C}{t^{2}%
}\cdot Var\left[y^{*}\left(0,t\right)\right]\label{Varw0}
\end{equation}
Using the solution formula, we also have
\begin{align}
w\left(x,t\right)=g'\left(y^{\ast}\left(x,t\right)\right)\text{ so }
Var\left[w\left(x,t\right)\right]=Var\left[g'\left(y^{\ast}\left(x,t\right)\right)\right]
\label{solform}
\end{align}
Combining (\ref{Varw0}) and (\ref{solform}) yields
\begin{align}
t^{2}Var\left[g'\left(y^{\ast}\left(0,t\right)\right)\right]=Var\left[y^{\ast}\left(0,t\right)\right].
\label{var1}
\end{align}
Since $g'$ is Brownian motion, one has $Var\left[g'\left(z\right)\right]=z$ for 
any $z>0$. Thus, one can formally write
\begin{align}
Var\left[g'\left(y^{\ast}\left(0,t\right)\right)\right]\sim\mathbb{E}y^{\ast}\left(0,t\right)
\label{var2}
\end{align}
Combining (\ref{var1}) and (\ref{var2}) then yields
\begin{align}
\mathbb{E}y^{\ast}\left(0,t\right)\sim\frac{C}{t^{2}}Var\left[y^{\ast}\left(0,t\right)\right]
\end{align}
Further, if we are willing to make the conjecture that
\begin{equation}
Var\left[y^{*}\left(0,t\right)\right]\sim\mathbb{E}\left[y^{*}\left(0,t\right)\right]^{p},
\end{equation}
for some power $p\in\mathbb{N}$, then we can deduce%
\begin{equation}
Var\left(  y^{\ast}\right)  =t^{2}Var\left(  g^{\prime}\left(  y^{\ast
}\right)  \right)  \sim t^{2}\mathbb{E}\left[  y^{\ast}\right]^{p}  \sim
Ct^{2}\cdot t^{p}.
\end{equation}
Then the result (\ref{Varw0}) can be further extended to obtain
\begin{equation}
Var\left(  w\left( 0,t\right)  \right)\sim Ct^{3p-2},
\end{equation}
which seems to formally suggest that the variance of the solution on $x=0$ 
increases as time increases. One can also perform computations using a 
finite difference scheme through a numerical study, e.g. from \cite{SS}.
\end{remark}

\subsection{Spatial density of shocks}

We now present a formal argument that the density of shocks in the solution
$w\left(  x,t\right)  $ in $x$ decreases as $t$ increases for the case
$H\left(  p\right)  =\left\vert p\right\vert $ with any random initial data. 
A shock can only form when there is a transition from Region I to II or III, 
or from Region II to III (see Figure \ref{Fig MS}).
As noted in \ref{Rmk Range}, the range of minimization problem is restricted
to the interval $\left\vert x-y\right\vert \leq t$, so that we have%
\begin{equation}
w\left(  x,t\right)  =\partial_{x}\min_{y\in\left[  x-t,x+t\right]  }g\left(
y\right)  =g^{\prime}\left(  y^{\ast}\left(  x,t\right)  \right)  .
\end{equation}
The minimum of $g\left(  y\right)  $ is either on the interior or attained at
one of the endpoints, $x\pm t$. If it occurs on the interior, we must have
$g^{\prime}\left(  y^{\ast}\right)  =0$. However, in the event of a minimum at
$x+t$, one may have $g^{\prime}\left(  y^{\ast}\right)  =g^{\prime}\left(
x+t\right)  <0$. The case $g^{\prime}\left(  x+t\right)  =0$ will occur on a
set of measure zero, and thus we have $w\left(  x,t\right)  <0$ for this case.
Similarly, if the minimum is attained at $y^{\ast}=x-t,$ we will have
$w\left(  x,t\right)  =g^{\prime}\left(  x-t\right)  >0$. We want to examine
probabilities and consider how minima can change with a small change in the
spatial coordinate.

Let time $t$ be fixed and consider a small change in $x$ from $x$ to $x+\Delta
x$. If the minimizer does not change, then the value of $w\left(  x,t\right)
$ also does not change by the argument above. If it does change, then so does
$g^{\prime}$, and consequently, $w\left(  x,t\right)  $. We calculate the
probability that there is a transition in $w$, i.e. the probability of a lower
minimum on $\left(  x+\Delta x-t,x+\Delta x+t\right)  $ compared to that on
$\left(  x-t,x+t\right)  $. This is given by%
\begin{equation}
\mathbb{P}\left\{  \min\left\{  g\left(  y\right)  :\left\vert x-y\right\vert
<t\right\}  >\min\left\{  g\left(  y\right)  :\left\vert x+\Delta
x-y\right\vert <t\right\}  \right\}  . \label{Sec3pt}%
\end{equation}

In order for (\ref{Sec3pt}) to be nonzero, there must be a minimum outside the
region of overlap, i.e. one of two events. In the first, a new minimum 
exists on the segment $\left(x+t,x+\Delta x+t\right)$, which we term region III, 
whose value is below the original minimum. In the second case, the original 
minimum fell in the range $\left(  x-t,x+\Delta x-t\right),$ (region I), 
which is excluded when moving from $x$ 
to $x+\delta x$ as illustrated in Figure \ref{Fig MS}. In
the latter case, the minimum attained on this segment fails to be in the new
interval when $x$ is increased by $\Delta x$, so the new minimum is a
consequence of the original no longer being contained in the domain of consideration.

In the former case (with the new minimum attained on the right), the minimum
moves lower, so $w\left(  x,t\right)  =\partial_{x}\min\left\{  g\left(
y\right)  \right\}  $ jumps from zero to a negative value. We express this
probability as%
\begin{align}
&  \mathbb{P}\left\{  w\left(  x_{-},t\right)  >w\left(  x_{+},t\right)  \right\}
\nonumber\\
&  =\mathbb{P}\left\{  \min\left\{  g\left(  y\right)  :\left\vert
x-y\right\vert <t\right\}  >\min\left\{  g\left(  y\right)  :x+t<y<x+\Delta
x+t\right\}  \right\}  \label{Sec3pt'}%
\end{align}
In the latter case, the minimum on the left gives rise to a larger minimum in
the central region, so the jump is to a minimum value.

The inequalities above are value for any $x$ and $\Delta x$. In articular, we
could write a set of inequalities for a partition $\left\{  x_{-2}%
,x_{-1},x_{0},x_{1},x_{2},...\right\}  $ of $\mathbb{R}$ with corresponding
$\Delta x_{i}:=x_{i}-x_{i-1}$.
We now perform the same computation for any $x$ but with time $t+\Delta t$ to
obtain the following:%
\begin{align}
&  \mathbb{P}\left\{  w\left(  x,t+\Delta t\right)  >w\left(  x,t+\Delta
t\right)  \right\}  \nonumber\\
&  =\mathbb{P}\left\{
\begin{array}
[c]{c}%
\min\left\{  g\left(  y\right)  :\left\vert x-y\right\vert <t+\Delta
t\right\}  >\\
\min\left\{  g\left(  y\right)  :x+t+\Delta t<y<x+\Delta x+t+\Delta t\right\}
\end{array}
\right\}  \label{Sec3dt}%
\end{align}

\begin{figure}
[htp]
\begin{center}
\includegraphics[width=\linewidth]%
{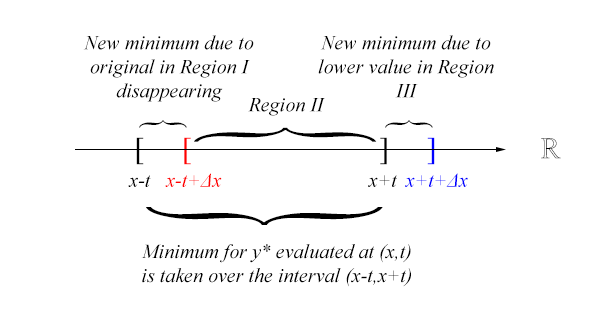}%
\caption{When moving from $\left(  x,t\right)  $ to $\left(  x+\Delta
x,t\right)  $, the interval of minimization is shifted $\Delta x$ to the
right. To have different minimizers at points $x$ and $x+\Delta x,$ at least
one of two events must occur: (i) the minimizer was attained in the leftmost
segment $\left(  x-t,x-t+\Delta x\right)  $ which is omitted for the
minimization problem at $x+\Delta x$ and hence lost; (ii) a new, lower minimum
is introduced by considering the segment $\left(  x+t,x+t+\Delta x\right)  $
on the right and hence gained.}%
\label{Fig MS}%
\end{center}
\end{figure}

\noindent We can now evaluate the probability in (\ref{Sec3dt}) to obtain the
probability of a transition for $\left(  x-\Delta t,t+\Delta t\right)  \,$,
i.e. obtain%
\begin{align}
&  \mathbb{P}\left\{  w\left(  x-\Delta t,t+\Delta t\right)  >w\left(
x_{+}-\Delta t,t+\Delta t\right)  \right\}  \nonumber\\
&  =\mathbb{P}\left\{
\begin{array}
[c]{c}%
\min\left\{  g\left(  y\right)  :x-t-2\Delta t<y<x+t\right\}  >\\
\min\left\{  g\left(  y\right)  :x+t<y<x+\Delta x+t\right\}
\end{array}
\right\}  .\label{Sec3dt''}%
\end{align}
We now compare (\ref{Sec3dt''}) with (\ref{Sec3pt'}) noting that the second
terms are identical. One has the inequality%
\begin{equation}
\min\left\{  g\left(  y\right)  :\left\vert x-y\right\vert <t\right\}
\geq\min\left\{  g\left(  y\right)  :x-t-2\Delta t<y<x+t\right\}  .
\end{equation}
In other words, in the notation of uniformly spaced points $\left\{
x_{i}\right\}  _{i\in\mathbb{Z}}$, we have%
\begin{equation}
\mathbb{P}\left\{  w\left(  x_{i},t\right)  >w\left(  x_{i+1},t\right)
\right\}  \geq\mathbb{P}\left\{  w\left(  x_{i+1},t+\Delta t\right)  >w\left(
x_{i+1}-\Delta t,t+\Delta t\right)  \right\}  \label{downshock}%
\end{equation}
and%
\begin{equation}
\mathbb{P}\left\{  w\left(  x_{i+1},t\right)  >w\left(  x_{i},t\right)
\right\}  \geq\mathbb{P}\left\{  w\left(  x_{i},t+\Delta t\right)  >w\left(
x_{i}+\Delta t,t+\Delta t\right)  \right\}  .\label{upshock}%
\end{equation}
By combining (\ref{downshock}) and (\ref{upshock}) and then summing over all
$i$, one sees that the density of jumps is decreasing in time, i.e.%
\begin{align}
& \sum_{i}\mathbb{P}\left\{  \text{jump between }x_{i}\text{ and }%
x_{i+1}\right\}  \\
& \geq\sum_{i}\left\{
\begin{array}
[c]{c}%
\mathbb{P}\left\{  w\left(  x_{i+1},t+\Delta t\right)  >w\left(
x_{i+1}-\Delta t,t+\Delta t\right)  \right\}  \\
+\mathbb{P}\left\{  w\left(  x_{i+1},t+\Delta t\right)  >w\left(
x_{i+1}+\Delta t,t+\Delta t\right)  \right\}
\end{array}
\right\}  .
\end{align}
Further, if the process is stationary, one has a stronger result in that the
local variation at a given $x_{i}$ is also nonincreasing. In the limit, one
may also write%
\begin{equation}
\mathbb{P}\left\{  w\left(  x,t\right)  >w\left(  x_{+},t\right)  \right\}
\geq\mathbb{P}\left\{  w\left(  x_{+},t+\Delta t\right)  >w\left(  \left[
x-\Delta t\right]  _{+},t+\Delta t\right)  \right\}
\end{equation}
and%
\begin{equation}
\mathbb{P}\left\{  w\left(  x,t\right)  <w\left(  x_{-},t\right)  \right\}
\geq\mathbb{P}\left\{  w\left(  x_{-},t+\Delta t\right)\right\}>w\left(\left[  x+\Delta t\right]
_{-},t+\Delta t\right).
\end{equation}
Thus, we can bound the probabilities that the right and left-hand limits 
of the solution at a particular point $x$ in space change as time is increased 
slightly for a transition of the type from Region I or II to III.

In addition to this case, there are potential transitions from Region I 
to II. In order for such a transition to occur, we would need
\begin{equation}
min\left\{g\left(y\right):x-t<y<x+\Delta x - t\right\} < min\left\{g\left(y\right)
: x+\Delta x -t<y<x+t\right\}.
\label{2to3i}
\end{equation}
Considering the same probabilities at $x+\Delta x$ and $t+\Delta t$, 
we have
\begin{align}
min\left\{g\left(y\right):x+\Delta t - \left(t+\Delta t\right)<y<x+\Delta t+\Delta x 
-\left(t+\Delta t \right)\right\} \nonumber\\
< min\left\{g\left(y\right):x+\Delta t+\Delta x - 
\left(t+\Delta t\right)<y<x+\Delta t+t+\Delta t\right\},
\end{align}
i.e.,
\begin{equation}
min\left\{g\left(y\right):x-t<y<x+\Delta x-t\right\}<min\left\{g\left(y\right):
x+\Delta x-t<y<x+t+2\Delta t\right\}.
\label{2to3ii}
\end{equation}

Then if a sample Brownian motion path $\omega$ satisfies \ref{2to3ii}, 
it will satisfy \ref{2to3i}, so one has
\begin{equation}
\mathbb{P}\left\{\text{Transition I}\rightarrow\text{II at }\left(x,t\right)
\right\}\geq\mathbb{P}\left\{\text{Transition I}\rightarrow\text{II at }
\left(x+\Delta x,t+\Delta t\right)\right\},
\end{equation}
and similarly, one shows that the density of shocks arising from 
these transitions is also decreasing in time.

\section{Application to a prototype case}

As discussed earlier, an important application of our methodology of involves the case where 
the flux function is given simply as $H\left(  p\right)  =\left\vert p\right\vert $.
We will first consider this case for (\ref{cl}) to illustrate the main ideas and a 
convergence proof before proceeding
to the more general case in the next section. As we shall see, though the
formulation of this problem is simple, it raises a number of interesting
points and the solution under our methods presents a clear geometric
interpretation. Even for an initial condition given by a stochastic process
such as Brownian motion, the essence of computing the solution profile is 
reduced to calculating in which of three regions Brownian
motion has a minimum. To calculate the solution $w(x,t)$, one needs to 
determine whether the minimum of the Brownian motion is at $x-t$, 
$x+t$, or within $(x-t,x+t) $. The more general case is 
based on the same principles but involves more calculation.

To ease notation, we will present the following prescription as to how one
determines the value of the solution profile $w\left(  x,t\right)  $ for a
particular realization of one Brownian path and for one specific point
$\left(  x,t\right)  \in\mathbb{R\times}(0,\infty)$. Implicit will be the
existence of an underlying, appropriate partition of the real line, the
technical details of which are omitted here and are to be outlined in the 
next section.

First note that the Legendre transform of the flux function $H$ given above
will be $0$ inside a compact interval and $+\infty$ elsewhere. More
specifically, we have%
\begin{equation}
tL\left(  \frac{y-x}{t}\right)  =\left\{
\begin{array}
[c]{c}%
0\\
+\infty
\end{array}
\right.
\begin{array}
[c]{c}%
y\in\left[  x-t,x+t\right]  \\
y\in\left(  -\infty,x-t\right)  \cup\left(  x+t,\infty\right)
\end{array}
.
\end{equation}
To this end, we consider the quantity%
\begin{equation}
tL\left(  \frac{y-x}{t}\right)  +g\left(  y\right)  .\label{qmin}%
\end{equation}

One then has the remarkable simplification that minimizing the
expression (\ref{qmin}) over the entire real line is equivalent to computing
the minimum of $g\left(  y\right)  $ in the interval $\left[  x-t,x+t\right]
$. Although finding the minimum of $g\left(  y\right)  $ for a stochastic
process, say integrated Brownian motion, may not be so simple, for our
purposes it will suffice to determine whether one of three distinct cases can
occur. Indeed, either the minimum of $g\left(  y\right)  $ will occur at one
of the endpoints or somewhere in the interior.

In our setup, we discretize Brownian motion in the following way. Without loss
of generality (see Remark \ref{BMextend}), consider the interval $\left[  0,1\right]  $ and let
$M\in\mathbb{N}$ be given. Divide the interval into $2^{M}$ pieces, resulting
in a partition $\left\{  r_{n}^{\left(  M\right)  }\right\}  $, and choose
random variables so that the discretized Brownian motion is constant on each subinterval,
with its values prescribed by a random walk. Our goal is to show that the 
solution with initial conditions given by the 
discretized version of the integrated Brownian motion indeed converges to
integrated Brownian motion as we let $M\rightarrow\infty$. We define our 
setup of the discretization more formally in Definition \ref{Def 1} before introducting 
our main results.

We will denote our discretized Brownian motion by $W_{M}$, and the 
integrated process as $I_{M}=\int W_{M}$. Note that our procedure 
superficially resembles the classical construction of Brownian motion as a 
limit of interpolated (and thus continuous) random variables. However, 
using the interpolated random walk would not be as useful in our 
case, as the minima of its integral will not general be at the specified 
$n$ points (vertices of $W_{M}$ and $L$).

We define
\begin{equation}
g_{M}^{\prime}\left(  r\right)  =W_{M}\left(  r\right)  :=W\left(
r_{i}\right)  \ \ \ for\ r\in\lbrack r_{i},r_{i+1}).
\end{equation}
We then integrate and define%
\begin{equation}
I_{M}\left(  t\right)  =\int_{0}^{t}W_{M}\left(  s\right)  ds,\label{disIBM}%
\end{equation}
where (\ref{disIBM}) defines our discretized integrated Brownian motion. The
initial value problem we want to consider then becomes (using superscripts 
for $M$ to ease notation)
\begin{equation}
w_{t}^{M}+\left\vert w^{M}\right\vert _{x}=0,\text{ }w^{M}\left(  x,0\right)
=g_{M}^{\prime}\left(  x\right)  .\label{Ninitial}%
\end{equation}
We construct a solution to this discretized problem. A key result, derived in
Section 3.1, is that%
\begin{equation}
w^{M}\left(  x,t\right)  =\left\{
\begin{array}
[c]{c}%
0\\
g_{N}^{\prime}\left(  x-t\right)  \\
g_{N}^{\prime}\left(  x+t\right)
\end{array}
\right.
\begin{array}
[c]{c}%
\text{if interior minimum}\\
\text{if minimum is left endpoint}\\
\text{if minimum is right endpoint}%
\end{array}
.\label{Nsol}%
\end{equation}
The events that the minimum occurs in these different cases (i.e., either
somewhere in the interior of $\left[  x-t,x+t\right]  $ or at one of the
endpoints) can be described as Gaussian integrals. In other words, the value
of the solution is completely determined by the value of the initial condition
depending on whether the minimum is attained at the left endpoint, in the 
interior, or at the right endpoint. The details are
presented in the next section for the more general case, but take the form of
a series of nested Gaussian integrals.

A second set of questions involve whether the solution converges, and if so 
in what sense to solutions to the limiting problem, i.e.
\begin{equation}
w_{t}+\left\vert w\right\vert _{x}=0,\text{ }w\left(  x,0\right)  =g^{\prime
}\left(  x\right)  \label{IBMinitial}%
\end{equation}
where $g^{\prime}$ is given by a Brownian motion. Indeed, 
one can show that we have convergence in probability of $w^{M}$ to $w$
in the following sense.

Given $A=B+C$, where $\mathbb{P}\left(\vert C \vert>\alpha\right)$ 
will be small, we will want to bound the probabilities of $\left\{ A>0 \right\}$ 
and $\left\{A<0\right\}$ in terms of sets only involving $B$. One may write

\begin{equation}
\mathbb{P}\left(B\leq-\alpha\right)-\mathbb{P}\left(C\geq\alpha\right)
\leq\mathbb{P}\left(A<0\right)\leq\mathbb{P}\left(B\leq\alpha)\right)
+\mathbb{P}\left(C\leq-\alpha\right).
\end{equation}

We use Billingsley Theorem 37.9, p. 534 \cite{BI}. Let $n\in\mathbb{M}$ be fixed 
and let $D$ be the dyadic rationals

\begin{equation}
D:=\left\{k2^{-M}:k\in\mathbb{M}\right\}
\end{equation}

Then we have for Brownian motion, $W$, with $\delta:=2^{-M}$, the 
bound for any $\alpha>0$
\begin{equation}
\mathbb{P}\left[\sup_{r\in\left[0,1\right]}\left\vert W\left(t+\delta \right)
-W\left(t \right)\right\vert >\alpha\right]\leq K\frac{\delta^{2}}{\alpha^{4}},
\end{equation}
where $K$ is a positive real number. By continuity of Brownian motion, 
we can extend this inequality to all $t$ (not just dyadic). We consider for 
simplicity the interval $\left[0,1\right]$ so $k\in\left\{1,2,...,2^{n}\right\}$ 
and let $Q:=2^{M}$, $\left(r_{0}, ..., r_{Q}\right)$ be a set of equally spaced 
points with $r_{0}:=0$ and $r_{N}:=1$ with spacing
\begin{equation}
\Delta r=\delta=2^{-M}.
\end{equation}
We have the, for any $r\in\left[0,1\right]$, the bound
\begin{equation}
\left\vert I\left(r\right)\right\vert\leq\sup_{s\in\left[0,1\right]}\left\vert
W\left(s\right)-W_{N}\left(s\right)\right\vert.
\end{equation}

Now define the following:
\begin{align}
\tilde{A}\left(s\right):=I\left(x-t\right)-I\left(s\right)\text{, }
\tilde{B}\left(  s\right)  :=I_{M}\left(  x-t\right)  -I_{M}\left(  s\right) \nonumber\\
\tilde{C}\left(s\right):=\left\{I\left(x-t\right)-I_{M}\left(x-t\right)\right\}
+\left\{I\left(s\right)-I_{M}\left(s\right)\right\}.
\end{align}
and the real numbers
\begin{align}
A:=\sup_{s\in\left[  x-t,x+t\right]  }\tilde{A}\left(  s\right), \ 
B:=\sup_{s\in\left[  x-t,x+t\right]  }\tilde{B}\left(  s\right),\nonumber\\
C:=\sup_{s\in\left[  x-t,x+t\right]  }\tilde{A}\left(  s\right)
-\sup_{s\in\left[  x-t,x+t\right]  }\tilde{B}\left(  s\right)\label{Def ABC}
\end{align}
Thus, we have $A=B+C$, and $C$ has the property
\begin{equation}
\vert C\vert\leq\sup_{\left(s\in x-t,x+t\right)}\vert\tilde{C}\left(s\right)\vert>\alpha
\end{equation}
Using the above, one can then write
\begin{equation}
\mathbb{P}\left\{\omega:\vert C\vert>\alpha\right\}\leq\mathbb{P}\left\{
\omega:\sup_{s\in\left[0,1\right]}\vert I\left(s\right)-I_{M}\left(s\right)\vert>\alpha\right\}
\leq K\frac{M^{-2}}{\alpha^{4}}
\end{equation}
and recalling that $\delta=\frac{1}{Q}$ by definition, we have
\begin{equation}
\mathbb{P}\left(B<-\alpha\right)-K\frac{M^{-2}}{\alpha^{4}}\leq\mathbb{P}\left(A<0\right)
\leq\mathbb{P}\left(B<\alpha\right)+K\frac{M^{-2}}{\alpha^{4}}
\end{equation}
Then, by choosing $\alpha$ and $N$ appropriately, one has
\begin{equation}
\mathbb{P}\left(  B<-M^{-\frac{2}{5}}\right)  -KM^{-\frac{2}{5}}%
<\mathbb{P}\left(  A<0\right)  <\mathbb{P}\left(  B<M^{-\frac{2}{5}}\right)
+KM^{-\frac{2}{5}}\label{absconv}%
\end{equation}
for a pure number $K$ independent of $N$. We use these inequalities 
in conjunction with the convergence theorems below.

\subsection{Properties of the Brownian motion discretization}

In particular, we want to prove that $I_{N}$ defined by (\ref{disIBM}) is a
Cauchy sequence in probability. Given an $\varepsilon>0$ we choose $J$ such
that $\Delta r_{N_{J}}\leq\varepsilon.$ Let $t\in\lbrack r_{i}^{\left(
J\right)  },r_{i+1}^{\left(  J\right)  }).$ This means that for all $j,k>J$ we
have, for any particular $\omega$ in the Polish space, the
inequality%
\begin{equation}
\left\vert W_{N_{j}}\left(  t\right)  -W_{N_{k}}\left(  t\right)  \right\vert
\leq\sup_{t_{1},t_{2}\in\lbrack r_{i}^{\left(  J\right)  },r_{i+1}^{\left(
J\right)  })}\left\vert W\left(  t_{1}\right)  -W\left(  t_{2}\right)
\right\vert .
\end{equation}
This means that for any $\omega$ and any real number $\alpha$, one has
\begin{align}
&  P\left\{  \omega:\left\vert W_{N_{j}}\left(  t\right)  -W_{N_{k}}\left(
t\right)  \right\vert >\alpha\right\} \nonumber\\
&  \leq P\left\{  \omega:\sup_{t_{1},t_{2}\in\lbrack r_{i}^{\left(  J\right)
},r_{i+1}^{\left(  J\right)  })}\left\vert W\left(  t_{1}\right)  -W\left(
t_{2}\right)  \right\vert >\alpha\right\}  \ .
\end{align}
A similar bound is obtained for the integrals of the discretized Brownian
motion. Using the definitions above we write, for $t\in\left[  0,1\right]  ,$
and any particular $\omega$ so that $W_{N_{j}}$ and $W_{N_{k}}$ are
continuous, bounded functions,%
\begin{align}
\left\vert I_{N_{j}}\left(  t\right)  -I_{N_{k}}\left(  t\right)  \right\vert
&  \leq\int_{0}^{t}\left\vert W_{N_{j}}\left(  s\right)  -W_{N_{k}}\left(
s\right)  \right\vert ds\nonumber\\
&  \leq\sup_{s\in\left[  0,1\right]  }\left\vert W_{N_{j}}\left(  s\right)
-W_{N_{k}}\left(  s\right)  \right\vert \nonumber\\
&  \leq\sup_{s\in\left[  0,1\right]  }\sup_{r\in\left[  0,\Delta r^{\left(
J\right)  }\right]  }\left\vert W\left(  s+r\right)  -W\left(  s\right)
\right\vert
\end{align}
so that for any $\omega$ and $\alpha\in\mathbb{R}$ we have
\begin{equation}
\left\vert I_{N_{j}}\left(  t\right)  -I_{N_{k}}\left(  t\right)  \right\vert
\geq\alpha\ \ \ implies\ \ \sup_{s\in\left[  0,1\right]  }\sup_{r\in\left[
0,\Delta r^{\left(  J\right)  }\right]  }\left\vert W\left(  s+r\right)
-W\left(  s\right)  \right\vert \geq\alpha\ .
\end{equation}
Hence we have%
\begin{align}
\left\{  \omega:\left\vert I_{N_{j}}\left(  t\right)  -I_{N_{k}}\left(
t\right)  \right\vert \geq\alpha\right\}   &  \subset\left\{  \omega
:\sup_{s\in\left[  0,1\right]  }\sup_{r\in\left[  0,\Delta r^{\left(
J\right)  }\right]  }\left\vert W\left(  s+r\right)  -W\left(  s\right)
\right\vert \geq\alpha\right\} \nonumber\\
P\left\{  \left\vert I_{N_{j}}\left(  t\right)  -I_{N_{k}}\left(  t\right)
\right\vert \geq\alpha\right\}   &  \leq P\left\{  \sup_{s\in\left[
0,1\right]  }\sup_{r\in\left[  0,\Delta r^{\left(  J\right)  }\right]
}\left\vert W\left(  s+r\right)  -W\left(  s\right)  \right\vert \geq
\alpha\right\}  . \label{I}%
\end{align}
Next, we claim that if $s$ and $r$ are in their respective closed, bounded
intervals, and one has for some $C\in\mathbb{R}$ and any fixed $\omega$ the
inequality%
\begin{equation}
P\left\{  \left\vert W\left(  s+r\right)  -W\left(  s\right)  \right\vert
\geq\alpha\right\}  \leq C
\end{equation}
then one also has%
\begin{equation}
P\left\{  \sup_{s\in\left[  0,1\right]  }\sup_{r\in\left[  0,\Delta r^{\left(
J\right)  }\right]  }\left\vert W\left(  s+r\right)  -W\left(  s\right)
\right\vert \geq\alpha\right\}  \leq C\ . \label{sup}%
\end{equation}
To prove this note that for each $\omega$ the Brownian motion $W\left(
s;\omega\right)  $ is continuous, so its difference $\left\vert W\left(
s+r\right)  -W\left(  s\right)  \right\vert $ is also continuous in $\left(
r,s\right)  $ and will have a maximum and minimum on $\left[  0,\Delta
r^{\left(  J\right)  }\right]  \times\left[  0,1\right]  $. Let $\left(
r^{\ast},s^{\ast}\right)  $ be a point of maximum. Then there is a sequence,
$\left(  r_{n},s_{n}\right)  $ converging to $\left(  r^{\ast},s^{\ast
}\right)  .$ We have thus%
\begin{equation}
\sup_{s\in\left[  0,1\right]  }\sup_{r\in\left[  0,\Delta r^{\left(  J\right)
}\right]  }\left\vert W\left(  s+r\right)  -W\left(  s\right)  \right\vert
=\left\vert W\left(  s^{\ast}+r^{\ast}\right)  -W\left(  s^{\ast}\right)
\right\vert .
\end{equation}
Let $R_{n}:=\left\vert W\left(  s_{n}+r_{n}\right)  -W\left(  s_{n}\right)
\right\vert $ and $R^{\ast}:=\left\vert W\left(  s^{\ast}+r^{\ast}\right)
-W\left(  s^{\ast}\right)  \right\vert .$ For each $\omega$ we have by
continuity in $s$ for $W\left(  s;\omega\right)  ,$ the pointwise convergence,
$R_{n}\rightarrow R^{\ast}.$ I.e., for any $\varepsilon>0$ we have convergence
with probability $1,$ also expressed as (B p. 75)
\begin{equation}
P\left[  \left\vert R_{n}-R^{\ast}\right\vert \geq\varepsilon\ i.o.\right]
=0.
\end{equation}

This implies convergence in probability, which implies convergence in
distribution, i.e., for any $\alpha$ such that $P\left[  R^{\ast}%
=\alpha\right]  =0,$ which in our case is all $\alpha$ by continuity, one has
\begin{equation}
\lim_{n\rightarrow\infty}P\left[  R_{n}>\alpha\right]  =P\left[
R>\alpha\right]  ,
\end{equation}
i.e., one has%
\begin{align}
\lim_{n\rightarrow\infty}P\left[  \left\vert W\left(  s_{n}+r_{n}\right)
-W\left(  s_{n}\right)  \right\vert >\alpha\right]   &  =P\left[  \left\vert
W\left(  s^{\ast}+r^{\ast}\right)  -W\left(  s^{\ast}\right)  \right\vert
>\alpha\right] \nonumber\\
&  =P\left[  \sup_{s}\sup_{r}\left\vert W\left(  s+r\right)  -W\left(
s\right)  \right\vert \right]  .
\end{align}
Hence, the assumption that $P\left[  \left\vert W\left(  s_{n}+r_{n}\right)
-W\left(  s_{n}\right)  \right\vert >\alpha\right]  \leq C$ implies the
conclusion $\left(  \ref{sup}\right)  ,$ proving the claim.

We use these in conjunction with a basic bound on Brownian motion to prove 
the following Lemma.
\begin{lemma}
For any $\alpha >0,$ one has 
\begin{equation}
P\left\{ \left\vert I_{N_{j}}\left( s\right) -I_{N_{k}}\left( s\right)
\right\vert >\alpha \right\} \leq 3\frac{\left\vert \Delta r^{\left(
J\right) }\right\vert }{\alpha ^{4}}\text{ \ for any }s\in \left[ 0,1\right]
.  \label{6}
\end{equation}
\label{lemmabound}
\end{lemma}

\begin{proof}
A random variable $X$ with finite kth moment satisfies Markov's 
inequality for any $k\in\mathbb{N}$
\begin{equation}
P\left\{  \left\vert X\right\vert >\alpha\right\}  \leq\frac{E\left[
\left\vert X\right\vert ^{k}\right]  }{\alpha^{k}}.
\end{equation}
We apply this to $X=W\left(  t\right)  -W\left(  s\right)  $ with $k:=4.$ Note
that $X$ is a Gaussian with mean $0$ and variance $\sigma^{2}=\left\vert
t-s\right\vert \ .$ The fourth moment of this Gaussian is then $3\sigma
^{4}=3\left\vert t-s\right\vert ^{2}.$

Thus, we can write, for $r\in\left[  0,\delta\right]  $ the inequality
\begin{equation}
P\left\{  \left\vert W\left(  s+r\right)  -W\left(  s\right)  \right\vert
>\alpha\right\}  \leq\frac{3\left\vert t-s\right\vert ^{2}}%
{\alpha^{4}}. \label{M}%
\end{equation}
From this inequality, one obtains
\begin{equation}
P\left\{  \sup_{s\in\left[  0,1\right]  }\sup_{r\in\left[  0,\Delta r^{\left(
J\right)  }\right]  }\left\vert W\left(  s+r\right)  -W\left(  s\right)
\right\vert >\alpha\right\}  \leq\frac{3\left[  \Delta r^{\left(  J\right)
}\right]  ^{2}}{\alpha^{4}}.
\end{equation}
Combining this with the inequality $\left(  \ref{I}\right)  ,$ for $j,k>J$ one
has for any $t\in\left[  0,1\right]  $%
\begin{align}
P\left\{  \left\vert I_{N_{j}}\left(  t\right)  -I_{N_{k}}\left(  t\right)
\right\vert >\alpha\right\}   &  \leq P\left\{  \sup_{s\in\left[  0,1\right]
}\sup_{r\in\left[  0,\Delta r^{\left(  J\right)  }\right]  }\left\vert
W\left(  s+r\right)  -W\left(  s\right)  \right\vert >\alpha\right\} \nonumber\\
&  \leq\frac{3\left[  \Delta r^{\left(  J\right)  }\right]  ^{2}}{\alpha^{4}}.
\end{align}
\end{proof}
Our next lemma entails passing to a subsequence to achieve the 
convergence in the desired sense. We state this without proof as it follows easily from analysis.
\begin{lemma}
There is a measurable function $I^{\ast }\left( \omega ,s\right) $
such that $I_{N_{j}}\rightarrow I^{\ast }$ in probability and a subsequence
converges $a.u.$ in $\omega $ and uniformly in $s.$\label{meassubseq}
\end{lemma}
With the definition%
\begin{equation}
B_{M}:=\sup_{s\in \left[ x-t,x+t\right] }\left\{ I_{M}\left( x-t\right)
-I_{M}\left( s\right) \right\} 
\end{equation}
one has from (\ref{absconv}) the bounds
\begin{align}
P\left\{ B_{M}<-M^{-2/5}\right\} -KM^{-2/5}&<P\left\{ A\leq 0\right\}\nonumber\\
&<P\left\{ B_{M}<M^{-2/5}\right\} +KM^{-2/5}.  \label{7}
\end{align}%
By Lemma \ref{meassubseq} we know that there is a subsequence of $\left\{ I_{N_{j}}\right\} 
$ converging a.u., although the convergence in
measure should be adequate. We consider only the subsequence
that converges almost uniformly in $\omega $ and also in measure in $\omega $
uniformly in the $s$ variable. We will just refer to this subsequence as $%
\left\{ I_{j}\right\} $ for brevity rather than $\left\{ I_{N_{j}}\right\} .$

By containment of sets in conjuction with (\ref{7}) for $j\in\mathbb{N}:$ $j^{-2/5}<\varepsilon$, one 
has the inequalities
\begin{eqnarray}
P\left( B_{j}<-\varepsilon \right) -Kj^{-2/5} &<&P\left\{
B_{j}<-j^{-2/5}\right\} -Kj^{-2/5}  \nonumber \\
&<&P\left\{ A\leq 0\right\}   \nonumber \\
&<&P\left\{ B_{j}<j^{-2/5}\right\} +Kj^{-2/5}<P\left( B_{j}<\varepsilon
\right) +Kj^{-2/5}  \label{8}
\end{eqnarray}

Define, analogous to $B_{M}$ above, the quantity
\[
B^{\ast }:=\sup_{s\in \left[ x-t,x+t\right] }\left\{ I^{\ast }\left(
x-t\right) -I^{\ast }\left( s\right) \right\} .
\]

Since Lemma \ref{lemmabound} says that $\left\{ I_{N_{j}}\left( \omega ,s\right) \right\} $
is fundamental in measure uniformly in $s$, we know that there is an $%
I^{\ast }$ to which $I_{N_{j}}$ converges in measure. Then we have
using $I_{j}$ in place of $I_{N_{j}}$ ,%
\[
I_{j}\left( x-t\right) -I\left( s\right) \rightarrow I^{\ast }\left(
x-t\right) -I^{\ast }\left( s\right) \text{ in prob, unif in }s.
\]%
Using Lemma \ref{meassubseq} we then have that 
\begin{eqnarray*}
\sup_{s}\left\{ I_{j}\left( x-t\right) -I\left( s\right) \right\} 
&\rightarrow &\sup_{s}\left\{ I^{\ast }\left( x-t\right) -I^{\ast }\left(
s\right) \right\} \text{ in prob, i.e .,} \\
B_{j} &\rightarrow &B^{\ast }\text{ in prob.}
\end{eqnarray*}%
Convergence in probability implies convergence in distribution, so that we
have%
\[
P\left\{ B_{j}<\alpha \right\} \rightarrow P\left\{ B^{\ast }<\alpha
\right\} .
\]%
Now using this in $\left( \ref{8}\right) $ we take limits and obtain,%
\begin{equation}
P\left\{ B^{\ast }<-\varepsilon \right\} \leq P\left\{ A\leq 0\right\} \leq
P\left\{ B^{\ast }<\varepsilon \right\} .  \label{9}
\end{equation}%
If we consider only the subsequence that converges almost uniformly, then we
have that for any $\delta >0$ we have a set $F$ such that $P\left( F\right)
<\delta $ and on $\Omega \ \backslash \ F$ one has uniform convergence of
continuous functions $B_{j}$ so that $B^{\ast }$ is continuous. Hence we can
take the limit as $\varepsilon \rightarrow 0$ and obtain from $\left( \ref{9}%
\right)$ that
\begin{equation}
P\left\{ B^{\ast }\leq 0\right\} \leq P\left\{ A\leq 0\right\} \leq P\left\{
B^{\ast }\leq 0\right\} .
\end{equation}
Thus we have the following theorem.
\begin{theorem}
In the limit $M\rightarrow\infty$, we have that $\mathbb{P}\left\{
w_{M}\left(x,t\right)=g_{M}^{\prime}\left(x-t\right)
\right\}\rightarrow\mathbb{P}\left\{w\left(x,t\right)
=g^{\prime}\left(x,t\right)\right\}$ for $1\leq i\leq 2N+1$. In other 
words, we have convergence in probability of the solution to the discrete 
problem to that of the continuous problem (\ref{cl}).
\end{theorem}

Recalling that $A<0$ is equivalent to $I\left(s\right)$ having its minimum at 
the left endpoint, i.e. (\ref{Def ABC}), we see that the probability that the minimum 
it attained at the left endpoint is a limit of probabilities involving the 
discretizations $I_{M}$.

While we have developed these ideas in the context of $H\left(p\right)=
\left\vert p\right\vert$, the same methodology can be implemented for 
general polygonal flux, as the latter simply intrdouces deterministic drift, and 
the bounds on Brownian motion remain valid.

\section{Analysis of the piecewise linear flux function with Gaussian
stochastic process initial data}

We now consider the flux function in
(\ref{cl}). To be precise, let the flux function $H$ be piecewise linear with
slopes given by $\left\{  m_{i}\right\}  _{i=1}^{N+1}$ in increasing order and
break points $\left\{  c_{i}\right\}  _{i=1}^{N}$. The Legendre transform $L$
will consequently have slopes $\left\{  c_{i}\right\}  _{i=1}^{N}$ and break
points $\left\{  m_{i}\right\}  _{i=1}^{N+1}$, with slope $-c_{N+1-i}$ in the
interval $I_{i}:=\left[  m_{i},m_{i+1}\right]  $, and infinite value outside
the interval $\left[  m_{1},m_{N+1}\right]  $. When evaluated at the argument
$\frac{x-y}{t}$, the function exhibits limiting behavior under small or large
time leading to interesting results in these limits for the solution of the
minimization problem. For small time $t$, the function $L$ is infinite outside
a small interval, making it more likely that a minimum is obtained at a vertex
of $L$. For large $t$, $L$ can be approximated as a single linear segment. In
this case the minimum is then likely to be at a vertex of $g$ rather than at a
vertex of $L$. In numerical studies, this can be illustrated by comparing the
variance in the spatial variable $x$ at given fixed times $t_{1}$, $t_{2}$,
etc. As is apparent in Figure \ref{Fig BM} computed using \cite{SS}, the total
variation tends to decrease as a function of $t$. This augments the formal
calculations shown in Section 3. This is despite the fact that the
probabilistic variance at a point $x$ increases in time.
\begin{figure}
[htp]
\begin{center}
\includegraphics[width=5.5in]%
{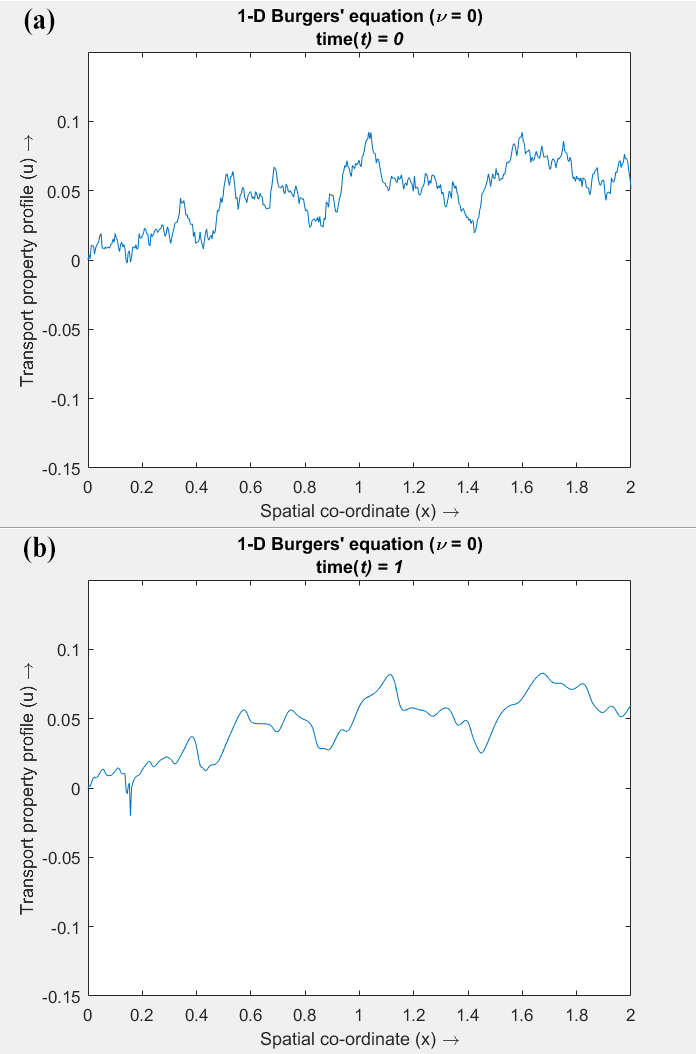}%
\caption{By performing numerical simulations using a finite difference scheme
on Burgers' equation for Brownian motion initial conditions, one can see that
the longer time behavior exhibits less variance in $x$ and appears more
smooth, in accordance with the analytic results.}%
\label{Fig BM}%
\end{center}
\end{figure}

In addition to the intuitive interpretation of the asymptotic limiting
behavior of the solutions, we prove a rigorous result for a broad class of
initial conditions. We consider a discretized approximation to a Gaussian
stochastic process as an initial condition paired with the piecewise linear
flux function. The virtue of this approach is that the result holds for any
stochastic process in this class, and is exact without relying on the
particular structure of say Brownian motion or Brownian bridge. We recall a 
result from \cite{CA}.

\begin{theorem}
\label{Thm 3.1}Let $L$ be polygonal convex, as described above, and $g^{\prime
}$ be piecewise constant. Then
\begin{align}
w\left(  x,t\right)   &  =g^{\prime}\left(  y^{\ast}\left(  x,t\right)
\right)  \text{ when the minimizer }y^{\ast}\text{ is at a vertex of
}L\nonumber\\
&  =L^{\prime}\left(  \frac{x-y^{\ast}\left(  x,t\right)  }{t}\right)  \text{
when the minimizer }y^{\ast}\text{ is at a vertex of }g \label{Thm 3.1a}%
\end{align}
a.e. in $x$ (for fixed $t>0$) is a solution to (\ref{cl}). Should a minimum 
occur at a point where both $g$ and $L$ have a vertex, one has a 
discontinuity in the solution with $w\left(x-,t\right)$ taking the value of the left 
limit $g'\left(x,t\right)$ and $w\left(x,t\right)$ having the value 
$g'\left(x+,t\right)$. Note that for a
fixed $t>0$ and $x$ a.e., one of the cases in (\ref{Thm 3.1a}) occurs.
\end{theorem}

This theorem will be needed as the final step in proving our main result of
this section. Before stating the theorem, we clarify the choice of discrete
points at which we will sample the stochastic process. This partition will 
depend on the value of $x$ and $t$. 
Given $x$ and $t$, the Legendre transform $L$ of the flux function will
only be finite on a finite interval, and therein $L^{\prime},$ the quantity of
interest, will take $N$ different values.  We partition each of these intervals 
into pieces at which we sample the values of $L'$ and $g'$. In 
addition, we will need to consider the value at each of the vertices of 
$L$. One is then left with a \textit{finite} number
of points, and computing the minimum amongst these is a far more manageable
task than calculating the minimum along a continuum. Indeed, our main result
presents an explicit formula for computing this minimum, and thus an exact
expression for the expectation value of the solution $w\left(  x,t\right)  $.
One can then consider refining such partitions and taking a limit to the continuum. 
In a computational context, it may be necessary to first fix a partition, and 
restrict oneself to that same grid, but we leave the technical details to 
future works.
\begin{definition}
\label{Def 1}
On the interval $\left[x-m_{N+1}t, x-m_{1}t\right]$ we set grid points 
for a fixed $x$ and $t$ as follows. 
Note that $L$ is divided into $N$ segments on which it takes finite values, and 
has $N+1$ associated vertices. Starting from the leftmost segment, we label 
these vertices as $r_{1,1}, ..., r_{1,n_{1}}, r_{2,1}, ..., r_{i,1}, ..., r_{N+1,1}$. We number the $n_{i}$ 
points that lie between each vertex sequentially using the second index, pairing 
them with the first index representing the vertex most immediately to its left, 
i.e. $r_{1,2}, ..., r_{1,n_{1}}$ for points between the first and second vertex. 
We denote the total number of points (vertices and points between) by $n$, 
and define the intervals $I_{i}^{x,t}:=(r_{i,1},r_{i+1,1})$.
\end{definition}

We illustrate the definition of these partitions in Figure \ref{Fig PI}.
\begin{figure}
[ptb]
\begin{center}
\includegraphics[width=\linewidth]
{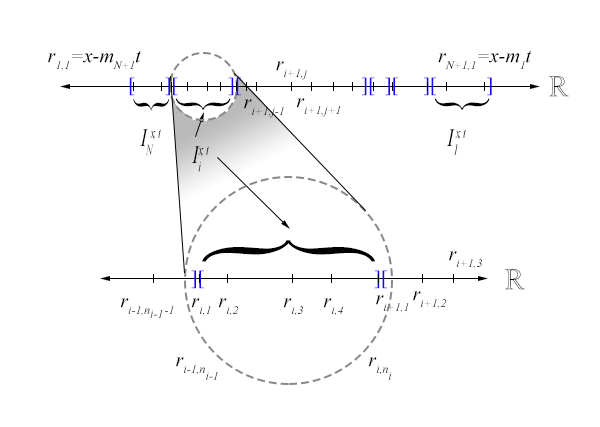}
\caption{Illustration of the partition defined in Definition \ref{Def 1}. Given 
a fixed $x$ and $t$, we split the interval $I^{x,t}=\left[x-m_{N+1}t,x-m_{1}t\right]$ 
into subintervals $I_{i}^{x,t}=\left[x-m_{i+1}t,x-m_{i}t\right]$ and partition 
each of these $N$ intervals into $n_{i}$ points. The problem of finding the 
minimum over the entire interval $I^{x,t}$ is then simplified to finding the 
minimum over these $\sum_{i=1}^{N}n_{i}=n$ points and $N+1$ vertices of 
$L$, which are the endpoints of $I_{i}^{x,t}$.}
\label{Fig PI}
\end{center}
\end{figure}

\begin{definition}
Let $x$ and $t$ be given and choose partitions of $I_{i}^{x,t}$ as described
in Definition \ref{Def 1}. Let the number of such points be denoted by $n_{i}%
$, and define $\sum_{i}n_{i}=:n$. Define $A_{i,j}$ as the inverse of the
covariance matrix of $X\left(  s\right)  $ with respect to these points
$\left\{  s_{i,j}\right\}  $ in the subpartition $\left \{r_{j,2}, ...,r_{j,n_{j}}\right \}$ and suppress the
second index of $A,s,$ etc. Denote the eigenvalues of $A_{i}$ as
$\tilde{\lambda}_{i}$, and define $\tilde{\mu}_{i}$ as the means of each
random variable, i.e.
\begin{equation}
\mu_{i,j}=\mathbb{E}\{g(r_{i,j}\}+
tL(\frac{x-r_{i,j}}{t}) \ for \ j\neq1, \ 1\leq i\leq N.
\end{equation}
The matrix $A_{i}$ can be diagonalized by%
\begin{equation}
A_{i}=U_{i}D_{i}U_{i}^{T}.
\end{equation}
Similarly, define $A,\left\{  \lambda_{i}\right\}  ,\left\{  \mu_{i}\right\}
$ etc. as the covariance matrix, eigenvalues, means of the random variables
for the entire process within the range $(x-t,x+t)$.

Define also $Q_{i}$ as follows:
\begin{equation}
Q_{i}=Q\left(i,x,t\right)  :=\min_{y\in I_{i}^{x,t}%
}\left\{  tL\left(  \frac{x-y}{t}\right)  +g\left(
y \right)  \right\}  \label{Thm1},%
\end{equation}
i.e. the minimum on the $ith$ segment of the flux function $L$.
\end{definition}

\begin{theorem}
\label{Thm 3.2}Let the flux function $H$ in the conservation law (\ref{cl}) be
piecewise linear and convex, with $N$ segments and $N+1$ 
vertices. Let the initial condition $g^{\prime}_N\left(
s\right)$ be given by a discretized Gaussian stochastic process, as
outlined in Definition \ref{Def 1}. The probability that the minimum 
occurs on the $i$th segment, i.e.
\begin{equation}
p_{i}:=\mathbb{P}\left\{  Q\left(  y^{\ast}\left(  x,t\right)  \right)
=Q\left(  y\left(  i,x,t\right)  \right)  \right\}\text{, i.e., }
w(x,t)=c_{N+1-i}  \label{Thm3}%
\end{equation}
is given by%
\begin{align}
p_{i} &  = \sum_{j=2}^{n_{i}}
\left(  2\pi\right)  ^{-\frac{n}{2}}\left\vert A\right\vert
^{\frac{1}{2}}\int_{-\infty}^{\infty}dx_{i,j}
{\displaystyle\prod\limits_{\left(m,l\right)\neq \left(i,j\right)}}\int_{x_{i,j}}^{\infty} dx_{m,l}
e^{-\frac{1}{2}\left(
\bar{x}-\tilde{\mu}\right)  ^{T}A\left(  \bar{x}-\tilde{\mu}\right)
}\nonumber\\
&for \ 1\le i \le N.\label{Thm4}%
\end{align}

The probability of such a minimum occuring at a vertex point 
of $L$ is given by an expression 
similar to (\ref{Thm4}) with the integral over the jth vertex as the outermost 
integral (treating the vertex as a segment with only one point).

Finally, define the event that the minimum occurs at the $jth$ vertex of 
$L$ by $R_{j}$ and set (for notational convenience) $p_{N+i}=\mathbb{P}
\left\{R_{i}\right\}$. Then the expected value of the solution is given by
\begin{align}
\mathbb{E}\left\{  w\left(  x,t\right)  \right\}  = \mathbb{E}
\left\{ L^{\prime}\left(
\frac{x-y^{\ast}\left(  x,t\right)  }{t}\right) \right\} \nonumber\\  
+ \sum_{i=1}^{N+1} \mathbb{P}\left\{ R_{i}\right\}
\mathbb{E}\left\{g'(y^{*}(x,t))\right\} \nonumber\\
=-\sum_{i=1}^{N}p_{i}c_{N+1-i}
+ \sum_{i=N+1}^{2N+1}p_{i}\mathbb{E}\left\{g'(y^{*}(x,t))\right\}\label{Thm5}%
\end{align}
a.e., that is, an average of the slopes of $L$ weighted by the probability of
the minimum occurring on the $i$th segment of $L$, plus terms at vertices of $L$.
\end{theorem}


\begin{remark}
In Theorem \ref{Thm 3.2}, we present closed-form probability 
of the solution taking on one of a number of values for a given 
$x$ and $t$, as well as the expectation of the solution at this 
given value. In fact, we can also write an expression for the 
distribution of the local minimum on segment $i$ as a function of its value 
$s$ by instead taking the last integral in (\ref{Thm4}) up to $s$
instead of $\infty$.

We may also write expressions for the cdf of the value of 
$g+L$ on the vertex, and for the cdf of the minimum value 
obtained on just one segment. Consider first, for
given $i\in\left[  1,N\right]  $, the domain $I_{i}^{x,t}$ (with parameters
$x,t$) where $L$ is a purely affine function. Then the cumulative probability
distribution (cdf) for the minimum on this $i$th segment, i.e. the quantity 
$Q_{i}$ is given by
\begin{equation}
F_{Q_{i}}\left(  s\right)  :=\left(  2\pi\right)  ^{-\frac{n_{i}}{2}%
}\left\vert A_{i}\right\vert ^{\frac{1}{2}}\int_{-\infty}^{s}...\int_{-\infty
}^{s}e^{-\frac{1}{2}\left(  \bar{x}-\bar{\mu}\right)  ^{T}A_{i}\left(  \bar
{x}-\bar{\mu}\right)  }dx_{i,n_{i}}...dx_{i,1}.\label{Thm2}%
\end{equation}
where $\bar{s}=\left(  s,s,...,s\right)  $ and $\bar{\mu}$ the vector given by
$\mathbb{E}\left\{  g\left(  \bar{y}\right)  +tL\left(  \frac{\bar{x}-\bar{y}%
}{t}\right)  \right\}  $ evaluated at the discrete set of points. 

Furthermore, the cdf of a potential 
minimum that occurs at the $jth$ of the $N+1$ vertices of $L$ is given by
\begin{equation}
F_{v_{j}}(s) = (2\pi)^{-\frac{1}{2}}\vert \mu_{j} \vert ^{-\frac{1}{2}}
\int_{-\infty}^{s}e^{-\frac{1}{2}\mu_{j}\left(x-\mu_{j}\right)^{2}
}dx.
\end{equation} 
\end{remark}


\begin{remark}
The case of the Gaussian stochastic process includes the deterministic case,
i.e. when $p_{i}=1$ for some $i$ dependent on $x$ and $t$ and $p_{j}=0$ for
$j\not =i$.
\end{remark}

\begin{proof}
[Proof of Theorem \ref{Thm 3.2}]We consider an initial condition given by the
discretized Gaussian stochastic process $g'_{\tilde{N}}(x)$. The key feature 
is that, using our methods, this problem of finding the minimum over the entire 
interval $(x-t,x+t)$ is reduced to finding a minimum 
over $n$ points, $n_{i}$ of which are on each segment, and $N+1$ points 
that are on the vertices of $L$. Given an $x$ and $t$, these $n$ points 
are fixed. Since we have a joint probability density 
for the values of this function at any of these points, it is a matter of 
integrating over the appropriate region(s) to obtain the desired probabilities. 
From there we can calculate not only the expectation value $\mathbb{E}
w(x,t)$ but the entire distribution of the solution $w$.
To this end, let $ \left\{  Y_{j}\right\}_{j=1}^{n}$ be a set of random 
variables with a multivariate probability density given by%
\begin{equation}
f\left(  x;\mu,A\right)  =\left(  2\pi\right)  ^{-\frac{n}{2}}\left\vert
A_{i}\right\vert ^{\frac{1}{2}}e^{-\frac{1}{2}\left(  \bar{x}-\bar{\mu
}\right)  ^{T}A_{i}\left(  \bar{x}-\bar{\mu}\right)  },
\end{equation}
where $A_{i}:=\sum_{i}^{-1}$ and $\sum_{i}$ is the covariance matrix of the
random variables in the subpartition $\zeta^{i}$, and $\bar{x},\bar{\mu}%
\in\mathbb{R}^{n_{i}}$ are vectors. Here the components of $\mu$ and $x$ are
related by $\mu_{i}=tL\left(  \frac{x_{i}-y_{i}}{t}\right)  +\mathbb{E}%
g\left(  y_{i}\right)  $. Thus, $f$ contains all the information we will need
to compute the minimum along the $i$th segment of $L$.

For simplicity, we write out the distribution over the points 
$n_{i,2}, ..., n_{i,n_{i}}$ to have values less than or equal to $s_{i}$. 
This is expressed as
\begin{align}
\mathbb{P}\left\{  Y_{2}\leq s_{1},...Y_{n_{i}}\leq s_{n_{i}}\right\}   &
=\mathbb{P}\left\{  Y_{i}\in%
{\displaystyle\bigotimes\limits_{i=1}^{n}}
(-\infty,s_{i}]\right\}  \nonumber\\
&  =\int_{-\infty}^{s_{1}}...\int_{-\infty}^{s_{n}}\frac{e^{-\frac{1}%
{2}\left(  \bar{x}-\tilde{\mu}\right)  ^{T}A_{i}\left(  \bar{x}-\tilde{\mu
}\right)  }}{\left(  2\pi\right)  ^{\frac{n_{i}-1}{2}}\left\vert \sum_{i}\right\vert
^{\frac{1}{2}}}dx_{n_{i}}...dx_{2}.\label{onesegprob}%
\end{align}
Note we have supressed the first index in $Y_{i,j}$ (which would be $i$) 
in these quantities over segment $i$ for ease of notation. To find the probability 
that a minimum on one segment is below a value $s$, we
can use the result (\ref{onesegprob}) and must all sum over all the (mutually 
exclusive) events that each random variable $Y_{i}$ is a minimum as follows
\begin{align}
\mathbb{P}\left\{  Y_{i}\in(-\infty,s)^{n_{i}}\right\} = 
\sum_{j=2}^{n_{i}} \left(  2\pi\right)
^{-\frac{n_{i}-1}{2}}\left\vert A_{i}\right\vert ^{\frac{1}{2}}\nonumber\\  
\left\{\int_{-\infty
}^{s}...\int_{-\infty}^{s}e^{-\frac{1}{2}\left(  \bar{x}-\tilde{\mu}\right)
^{T}A_{i}\left(  \bar{x}-\tilde{\mu}\right)  }dx_{n_{i}}...dx_{1}\right\}
,\label{exprmaxall}%
\end{align}
where we have $(n_{i}-1)-$fold integrals from $-\infty$ to $s$ on the right-hand
side of (\ref{exprmaxall}). A similar construction then yields the result (\ref{Thm4}).

We now want to use equation (\ref{Thm 3.1a}) from Theorem \ref{Thm 3.1},
which states the solution $w\left(  x,t\right)  $ takes the value of
$L^{\prime}$ at any point where the minimum is achieved at a vertex of $g$, 
and the value $g$ when achieved at a vertex of $L$. By accounting for both the
possibility of a minimum at one of the $N$ segments or at one of the vertices 
of $L$, we can write not only the expectation value of the solution at a given 
point, but the entire distribution. The expectation value is consequently
\begin{align}
\mathbb{E}\left\{  w\left(  x,t\right)  \right\}  = \mathbb{E}
\left\{ L^{\prime}\left(
\frac{x-y^{\ast}\left(  x,t\right)  }{t}\right) \right\} \nonumber\\  
+ \sum_{j=1}^{N+1} \mathbb{P}\left\{R_{j}\right\}
\mathbb{E}\left\{g'(y^{*}(x,t))\right\} \nonumber\\
=-\sum_{i=1}^{N}p_{i}c_{N+1-i}
+ \sum_{i=N+1}^{2N+1}p_{i}\mathbb{E}\left\{g'(y^{*}(x,t))\right\},
\end{align}
i.e., it is given by the average over these possibilities. The event of a minimum 
at a vertex of $L$ may be more than a set of measure zero at the discrete level, 
but one might expect convergence to zero as we take the limit 
$\tilde{N}\rightarrow \infty$. The situation in regards to these two cases
 is illustrated in Figure \ref{Fig MinL}.
\end{proof}
\begin{figure}
[h]
\begin{center}
\includegraphics[width=\linewidth]%
{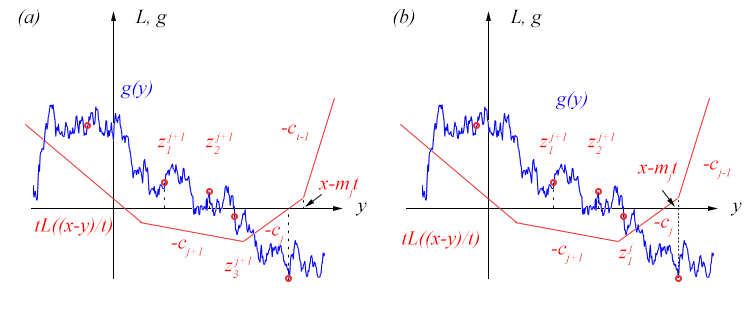}%
\caption{We are seek the minimum of $L\left(  \frac{x-y}{t}\right)  +g\left(
y\right)  $ amongst a discrete, finite set of points $\left\{  z_{i}%
^{j}\right\}  $ which are chosen based on a fixed $x$ and $t$. Here we display
integrated Brownian motion, but these arguments hold for \textit{any
}stochastic process: (a) In most cases, none of the points $\left\{  z_{i}%
^{j}\right\}  $ coincide with the vertices of the function $L$, so the minimum
on each segment is well-defined and one computes $\mathbb{E}\left\{
w\right\}  $ as described below; (b) Occasionally the points $\left\{
z_{i}^{j}\right\}  $ may intersect $\left\{  x-m_{k}t\right\}  $ for some
$i,j,k$ but this is a set of measure zero and can be ignored.}%
\label{Fig MinL}%
\end{center}
\end{figure}
For a specific
realization of the stochastic process $g_{\tilde{N}}\left(  x\right)  $, the values of the
solution $w$ are within the range $\left\{  c_{i}\right\}  $. The solution $w$
takes the value $c_{N+1-i}$ when the minimum is achieved along the $i$th
segment. If the minimum is achieved at a vertex of $L$, the expectation of the 
solution is zero by assumption, so the second set of terms drops out.

\begin{remark}
As an aside, we note that one can transform the integrals to a different
coordinate system and simplify the exponentials to only square terms. However,
the domain of integration is transformed in such a way that the integration is
nontrivial and results in a nested sequence of integrated error functions. By
construction, the matrix $A$ is symmetric and real, so the eigenvalues are
real there exists a matrix $U_{i}$ such that $U^{T}AU=D$, where $D$ is the
diagonal matrix consisting of the eigenvalues $\lambda_{n_{i}}$. Any repeated
eigenvalues are assigned to $\left\{  \lambda_{i}\right\}  $ as needed, i.e.
$\lambda_{k}=\lambda_{k+1}$ is permitted. Furthermore, we know $U^{-1}=U^{T}$,
so we can write%
\begin{equation}
\left(  \bar{x}-\bar{u}\right)  ^{T}A\left(  \bar{x}-\bar{\mu}\right)
=\left(  \bar{x}-\bar{\mu}\right)  UDU^{-1}\left(  \bar{x}-\bar{\mu}\right)
=z^{T}Dz=\sum_{i=1}^{n}\lambda_{i}z_{i}^{2}%
\end{equation}
where%
\begin{equation}
\bar{z}=U^{-1}\left(  \bar{x}-\bar{\mu}\right)
\end{equation}
for a unitary matrix $U$ so that $\left\vert U\right\vert ^{2}=1$. Hence,
given any vector $\mu$ and a real, symmetric matrix $A$, we can write%
\begin{equation}
-\frac{1}{2}\left(  \bar{x}-\bar{\mu}\right)  =-\frac{1}{2}\sum_{i=1}^{n_{i}%
}\lambda_{i}z_{i}^{2}.
\end{equation}
Thus, one has%
\begin{equation}
f\left(  z_{j},\left\{  \lambda_{j}\right\}  \right)  =\left(  2\pi\right)
^{-n/2}\left\vert \Sigma\right\vert ^{-\frac{1}{2}}e^{-\frac{1}{2}\sum
_{i=1}^{n_{i}}\lambda_{i}z_{i}^{2}}%
\end{equation}
or equivalently%
\begin{equation}
\mathbb{P}\left\{  Z_{1}\leq z_{1},...,Z_{n_{i}}\leq z_{n_{i}}\right\}
=\int...\int_{\Omega}\frac{e^{-\frac{1}{2}\sum_{i=1}^{n}\lambda_{i}\tilde
{z}_{i}^{2}}}{\left(  2\pi\right)  ^{\frac{n}{2}}\left\vert \sum\right\vert
^{\frac{1}{2}}}d\tilde{z}_{n_{i}}...d\tilde{z}_{1}\label{coordChangeOneSeg}%
\end{equation}
over the transformed domain $\Omega$ (rotated and scaled from the original
region $\left[  -\infty,s\right]  ^{n_{i}}$).
\end{remark}

\section{Exact and approximate results for the density of shocks}

We now consider the density of shocks in the solution $w\left(
x,t\right)  $, i.e. a spatial coordinate $x$ where
\begin{equation}
w\left(  x_{-},t\right)  \not =w\left(  x_{+},t\right)  .
\end{equation}
This was discussed in the special case $H\left(  p\right)  =\left\vert
p\right\vert $ in Section 3. Recall from Section 4 the notation $I_{i}^{x,t}
=(r_{i,1},r_{i+1,1})$
to track the intervals along which $L$ has slope $d_{i}$. 
From the arguments above, it is clear that at point $x$ 
where a shock occurs, the minimizer $y^{\ast}\left(  x,t\right)$ jumps 
from one segment of $L$ to another. By
continuity, we can assume that for sufficiently small $\varepsilon>0$, $y^{\ast}\left(
x-\varepsilon,t\right)  $ is located within the same segment in the limit. In other words, one writes%
\begin{equation}
\left\{y^{\ast}\left(
x-\varepsilon,t\right)\text{ }|\ \varepsilon<\varepsilon_{0} \right\}  \subset I_{i}^{x,t}%
\backslash\partial I_{i}^{x,t}
\end{equation}

Let $\Delta x>0$ and consider the following. Suppose that $y_{1}:=y^{\ast
}\left(  x-\Delta x,t\right)  $ is located where $L$ has slope $d_{1}$ and
$y_{2}:=y^{\ast}\left(  x,t\right)  $ where $L$ has slope $d_{2}$. When we
shift the spatial coordinate by amount $\Delta x$, the value of the quantity
to be minimized changes by $R_{1}=-d_{1}\Delta x$ at $y_{1}$ and by
$R_{2}=-d_{2}\Delta x$ at $y_{2}$. Thus, the net change is a decrease $\left(
d_{2}-d_{1}\right)  \Delta x$ at $y_{2}$ relative to $y_{1}$. This means a
shock from this limit can only occur when $d_{2}>d_{1}$. This argument is
illustrated in Figure \ref{Min2Segments}. If we let $\Delta x$ be negative,
i.e. we consider the points $x$ and $x+\Delta x$ instead of $x-\Delta x$ and
$x$, we can have a shock for the case $d_{1}>d_{2}$. We now want to find an
expression for the probability density of such shocks.

To be more precise, we let $a:=\Delta x\left(  d_{2}-d_{1}\right)  .$ We want
to compute the probability that the minimum shifts from one segment to
another. 
First consider the simpler problem with only two segments, and the
probability of the minimum shifting from a particular point on the first
segment to a particular point on the second segment. To illustrate this in a 
simple way, consider the points $r_{1,2}$ and $r_{2,2}$ among the $n_{1}+n_{2}+1$
points in the system. We also use the notation that the random variable 
$Y_{i,j}$ corresponds to the point $r_{i,j}$, supressing the subscript $r$.
 We compute the probability that the
variable $Y_{1,2}$ associated with the point $r_{1,2}$ is at least $s$,
$Y_{2,2}$ is in the range between $s$ and $s+a$, and $Y_{1,i}$, 
$Y_{2,i}$, and $Y_{3,1}$ are all at least $s+a$ for all $i\neq 2$ as follows
\begin{align}
&  \mathbb{P}\left\{  Y_{1,2}=s_{1,2},Y_{2,2}\in\left(  s_{1,2},s_{1,2}+a\right)
,Y_{j,i}\geq s_{1,2}+a\text{ for all }i\neq 2, 1\leq j \leq 3
\right\}  \nonumber\\
& = \int_{s_{1,2}}^{s_{1,2}+a}ds_{2,2}
\displaystyle\prod\limits_{\left(m,n\right)\neq\left(1,2\right)\ or \ \left(2,2\right)}
\int_{s_{1,2}+a}^{\infty}ds_{m,n}
f\left(s_{1,1},s_{1,2},...,s_{3,1}\right)
\end{align}
where $f$ is the joint distribution for the variables $s_{1,1},...,s_{3,1}$ 
at the $n_{1}+n_{2}+1$ points, and we have fixed $s_{1,2}$. Now, we want to
divide by $a$ and take the limit. When we differentiate with respect to $x$, the integral with respect to
$s_{2,2}$ (in the $(n_{1}+1)-$th argument of $f$) drops out due to the
Fundamental Theorem of Calculus, assuming sufficient continuity properties. The 
probability of transitioning from a minimum from one segment to the other can 
be written explictly as follows
\begin{align}
\mathbb{P}_{\left(1,2\right),\left(2,2\right)}:&=
\lim_{\Delta x \rightarrow0}\frac{1}{\Delta x}\mathbb{P}\left\{
\text{min at }\left(1,2\right)\text{ at }\left(x,t\right)\text{ and min at }
\left(2,2\right)\text{ at }\left(x+\Delta x,t\right)\right\}\nonumber\\
&=\left(  d_{2}-d_{1}\right)\int_{-\infty}^{\infty}ds_{1,2}
{\displaystyle\prod\limits_{(u,v) \neq  (1,2), (2,2)}}ds_{u,v}
\int_{s_{1,2}}^{\infty}
f\left(s_{1,1}, s_{1,2}, ..., s_{2,1},s_{2.2}, ..., s_{3,1},
\right).
\end{align}
i.e. the $s_{2,2}$ entry is evaluated at $s_{1,2}$ in $f$.
This can be generalized to computing the probability of transition from one
segment to another, which will entail $n_{i}n_{j}$ terms. These terms
correspond to a jump from a minimum at $s$ 
attained at one of the $n_{i}$ possible points from
segment $i$ to one of the $n_{j}$ possible points on segment
$j$. Now we extend this computation from one point to another, 
to all points in a segment $i$ to another segment $j$. Let
\begin{align}
\tilde{P}_{i,j}\left(  s; a\right)  :=\mathbb{P} \Big\{  & \min_{2\leq k\leq
n_{i},}Y_{i,k}=s,\text{ }\min_{2\leq k\leq n_{j}}
Y_{j,k} \in\left(s,s+a\right), \nonumber\\
& \min_{m\neq i,j \ or \ n=1}Y_{m,n}\geq s+a\Big\},\label{switchsegments}%
\end{align}
that is, the probability that the minimum on segment $i$ equals
 a given value $s$, that the minimum on a different segment
$j$ falls slightly above this same value $s$, and that the minimum on 
all of the other segments excluding $i$ and $j$ is above the value 
$s+a$. Then, as $x$ is changed
an infinitesimal amount, the value of the solution $w\left(  x,t\right)  $
will flip from $d_{1}$ to $d_{2}$. In the limit $a\downarrow0$ (or 
equivalently, $\Delta x \downarrow0$), the density of
shocks is obtained. One may write an exact expressiom and take 
this limit to obtain

\begin{figure}
[htp]
\begin{center}
\includegraphics[width=\linewidth]%
{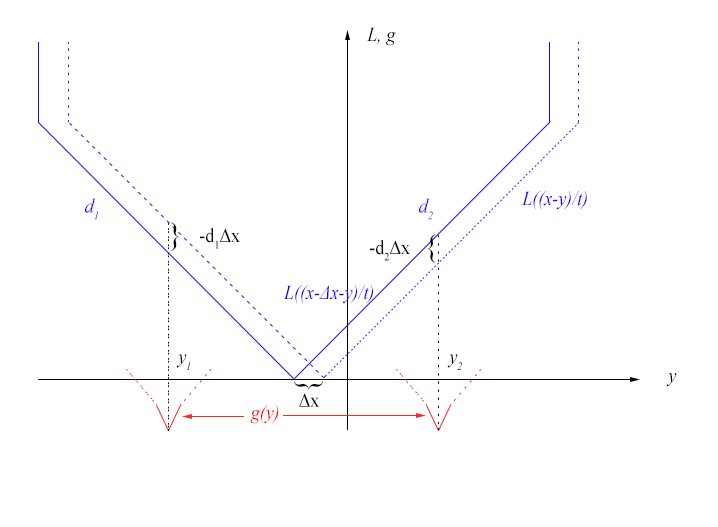}%
\caption{We are interested in conditions under which the global minimizer
$y^{\ast}\left(  x,t\right)  $ changes from one segment to another, as this is
the only case for which the solution profile $w\left(  x,t\right)  =L^{\prime
}\left(  y^{\ast}\left(  x,t\right)  \right)  $ will change, resulting in a
shock. Consider for ease of notation the case for which $L$ has only two
slopes and take an initial condition $g$ that has local minima at $y_{1}$ and
$y_{2}$ and rapidly increases elsewhere. We sketch $L\left(  \frac{x-\Delta
x-y}{t}\right)  $ and $L\left(  \frac{x-y}{t}\right)  $, and want the minimum
to change in the limit $\Delta x\rightarrow0$. For simplicity, in this
illustration we consider the case $N=2$, i.e. where $L$ has only two slopes.
In order for the minimum to switch from $y_{1}$ to $y_{2}$, the relative
change of the argument at $y_{2}$ compared to $y_{1}$, given by $\left(
d_{2}-d_{1}\right)  \Delta x$, must be negative.}%
\label{Min2Segments}%
\end{center}
\end{figure}

\begin{align}
\mathbb{P}_{i,j}:&= \lim_{\Delta x \rightarrow 0}\frac{1}{\Delta x}
\mathbb{P}\left\{w\left(x,t\right)=d_{i}\text{ and }w\left(x+\Delta x,t\right)=d_{j}\right\} \nonumber\\
&=\sum_{k=1}^{n_{i}}\sum_{k^{\prime}=1}
^{n_{j}}\left(  d_{k^{\prime}}-d_{k}\right)\int_{-\infty}^{\infty}ds
{\displaystyle\prod\limits_{(m,n) \neq  (i,k), (j,k')}}\nonumber\\
& \int_{s}^{\infty}ds_{m,n}f\left(s_{1,1}, ..., s_{N+1,1}
\right)|_{s_{i,k}=s_{j,k'}=s}.
\label{switchsegmentsall}%
\end{align}
for the density of any transitions of a minimum between segments $i$ and
$j$. The joint probability distribution function $f$ is as outlined in (\ref{Thm2}).


One can diagonalize this matrix. Indeed, denoting the eigenvalues of
$A$ by $\left\{  \lambda_{m}\right\}  _{m=1}^{n}$, 
$A$ can thus be diagonalized by $U^{T}A^
U=D$, where the diagonal matrix $D$ is given by
$D _{jk}=\delta_{jk}\lambda_{j}$. We may then
rewrite the integrand $f\left(  \cdot\right)  $ in (\ref{switchsegmentsall}) as $\prod_{l=1}^{n-2}%
e^{-\lambda_{l}z_{l}^{2}}$ and have, similar to the proof of Theorem
\ref{Thm 3.2}:%
\begin{equation}
P_{i,j}\left(  s\right)  =\left(  2\pi\right)  ^{-\frac
{n}{2}}\sum_{k=1}^{n_{i}}\sum_{k^{\prime}=1}^{n_{j}
}\left(  d_{k^{\prime}}-d_{k}\right)  \left\vert A\right\vert
^{\frac{1}{2}}\int_{\Omega_{(i,k),(j,k')}}%
{\displaystyle\prod\limits_{l=1}^{n-2}}
e^{-y_{l}^{2}}dy_{l}\label{transitionDensityas}%
\end{equation}
where the transformed domains of integration depends upon $k$ and 
$k^{\prime}$. The density of transitions, in a range $\left(  s,s+ds\right)
$ can be expressed using these results. The integration over the domain $\Omega$ 
may appear complicated due to the high dimension of the space and the 
nested nature of the integrals, but one can make some simplifications. 
For example, as is further illustrated in the next section, the eigenvalue spectrum 
of $A$ falls off rapidly. Thus, only a (smallest) few of the eigenvalues are relevant and many 
of the integrals can be well-approximated as $\delta-$functions about 
$y_{l}=0$.

\subsection{Transitions between vertex and non-vertex points}

Above, we have considered the question of transition from an overall 
minimum from a segment $i$ to a different segment $j$. There is 
also the issue of a minimum obtained at one of the $N$ segments 
of $L$ replaced by a minimum at a vertex, and vice versa. We will 
quantify the probability of such an occurence.

First, we comment that without discretization, we would 
have a transition from $w(x,t)=d_{i}$ to $w(x,t)=g'(j,1)$ if $Y_{i,k}=s$ 
and $Y_{j,1}\in(s,s-\Delta x (d_{i}+g'(j,1))$. Now, with the discrete model, 
we will evaluate exactly the probability that if $Y_{i,k}=s$, then $Y_{j,1}
\in(s,s-\Delta x (d_{i}+g'(j,1)))$ and vice versa. We assume for the 
first case that $d_{i}+g'\left(j,1\right)<0$.

Recall that the grid we place is for a fixed $x$ and $t$. In this section, 
we will consider changing $x$ (or similarly, $t$) by such a small amount 
$\Delta x$ that all of the original grid points within $(x-t,x+t)$ remain, and no new 
grid points are introduced into the interval $(x+\Delta x -t, x+\Delta x +t)$. 
Thus, the points $r_{k,l}$ remain unchanged. We also assume that the minimum 
attained at $(x,t)$ occurs at the vertex $r_{j,1}$ and make the approximation 
that $L$ still nearly has a vertex at $r_{j,1}$ despite the $\Delta x$ shift.

The first step is to identify the condition for local minima at vertices $r_{j,1}$ 
to overtake $Y_{i,k}$, the absolute minimum over $\mathbb{R}$ located at 
some $k$th point on the $i$th segment. We use $Y_{i,k}^{new}$ to denote 
the new local minimum when $L$ is shifted by $\Delta x$. We calculate 
\begin{align}
Y_{i,k}^{new}:=Y_{i,k}-\Delta xd_{i}>Y_{j,1}+g'(r_{j,1})\Delta x =: Y_{j,1}^{new}
\label{gToL}
\end{align}

In other words, (\ref{gToL}) is the requirement for a transition from a minimum 
at the vertex of $g$ at $r_{i,k}$ to the vertex of $L$ originally at $r_{j,1}$ that 
has now shifted. Of course, we also need the condition
\begin{align}
Y_{i,k}<Y_{j,1}\label{cond1}
\end{align}

The condition (\ref{gToL}) is equivalent to
\begin{align}\label{cond2}
Y_{i,k}-\Delta x(d_{i}+g'(r_{j,1}))>Y_{j,1}
\end{align}
Once we have established conditions (\ref{cond1}$-$\ref{cond2}), we can now calculate 
that probability using the discrete setup in which we deal with only the points 
${r_{i,1}, r_{i,2}, ...}$, i.e. perform the calculation using Gaussian integrals 
as follows. Let $\mathbb{P}_{i,j}^{\Delta x}$ be defined as the probability of a 
transition from segment $i$ of $L$ to the vertex point $r_{j,1}$ of $L$, so that

\begin{align}
\mathbb{P}_{i,j}^{\Delta x}  &:=
\sum_{k=1}^{n_{i}}\int_{-\infty}^{\infty}ds\mathbb{P}\left\{
Y_{i,k}=s, \ Y_{j,1}\in\left(s,s-\left(d_{i}+g'\left(r_{j,1}\right)\right)\Delta x\right)\right\} \nonumber\\
&=\sum_{k=1}^{n_{i}}\int_{-\infty}^{\infty}ds\int_{s}^{s-\left(d_{i}
+g'\left(r_{j,1}\right)\right)\Delta x}dx_{j,1}\nonumber\\
&{\displaystyle\prod\limits_{\left(m,n\right)\neq\left(i,k\right),\left(j,1\right)}}
\int_{s}^{\infty}dx_{m,n}
f\left(x_{1,1},...x_{N+1,1}\right)\vert_{x_{i,k}=s}\nonumber\\
\mathbb{P}_{i,j}  &:= \lim_{\Delta x \rightarrow0} \frac{\mathbb{P}_{i,j, \Delta x}}{\Delta x}
=-\left(d_{i}+g'\left(r_{j,1}\right)\right)\sum_{k=1}^{n_{i}}
\int_{-\infty}^{\infty}ds\nonumber\\
&{\displaystyle\prod\limits_{\left(m,n\right)\neq
\left(i,k\right),\left(j,1\right)}}\int_{s}^{\infty}dx_{m,n}
f\left(...\right)\vert_{x_{i,k}=s,x_{j,1}=s}\nonumber\\
&=\mathbb{P}\left\{w\left(x_{-},t\right)=d_{i}\text{ and }w\left(x_{+},t\right)
=g'\left(r\left(j,1\right)\right)\right\}.
\end{align}

The calculation for the opposite transition, which can only occur when $d_{i}+
g'\left(r\left(j,1\right)\right)$ is positive, is
\begin{align}
\mathbb{P}_{i,j}:&=\mathbb{P}\left\{w\left(x_{-},t\right)=g'\left(r\left(j,1\right)\right)
\text{ and }w\left(x_{+},t)=d_{i}\right)\right\}\nonumber\\
&=\lim_{\Delta x\rightarrow0}\frac{1}{\Delta x}\sum_{k=1}^{n_{i}}\int_{-\infty}^{\infty}
ds\mathbb{P}\left\{Y_{j,1}=s,Y_{i,k}\in\left(s,s+\left(d_{i}+g'\left(j,1\right)\right)\Delta x\right)
\right\}\nonumber\\
&=\left(d_{i}+g'\left(r_{j,1}\right)\right)\int_{-\infty}^{\infty}ds
{\displaystyle\prod\limits_{\left(m,n\right)\neq\left(i,k\right),\left(j,1\right)}}
\int_{s}^{\infty}dx_{m,n}f\left(...\right)\vert_{x_{i,k}=x_{j,1}=s}.
\end{align}

\section{Applications to selected examples and open problems}

The results of the preceding sections apply in broad generality for any kind
of randomness that falls under the umbrella of Gaussian stochastic processes.
In addition to the general results proven for an arbitrary Gaussian stochastic
process, it is also interesting to analyze the behavior of the solution to the
conservation law under a few specific cases. In particular, we will illustrate
the results for the cases of Brownian motion and Brownian bridge. As these
processes correspond to the initial condition $g^{\prime}$, we must also note
that $g\left(  x\right)  $ will be a Gaussian stochastic process for both of
these cases, termed integrated Brownian motion and integrated Brownian bridge,
respectively. For each of these cases, the covariance matrix takes a special
form as was described in Section 2.

\subsection{Brownian motion and Brownian bridge as initial conditions}

In cases such as Brownian motion, Brownian bridge, or Ornstein-Uhlenbeck, 
one observes specific structure in the matrices and integrals entailed in 
potential applications of Theorem \ref{Thm 3.2}, which may make it somewhat 
more tractable to numeric computation than an arbitrary Gaussian stochastic 
process. For example, in some cases the relevant inverses of covariance
matrices can be neglected aside from a narrow band of elements close to the
diagonal. For example, consider an integrated Brownian motion $S\left(
t\right)  $ and let $\left\{  t_{i}\right\}  _{i=1}^{n}$ be given. Observe
that the covariance matrix $\sum$ for integrated Brownian motion from
(\ref{covMatrixIBM}) can be written as \cite{RO}%
\begin{equation}
\sum\nolimits_{ij}=\frac{1}{N^{3}}\min\left(  i,j\right)  ^{2}\left(
\frac{\max\left(  i,j\right)  }{2}-\frac{\min\left(  i,j\right)  }{6}\right),
\end{equation}
so it is symmetric with respect to permuting $i$ and $j$.
One of the main quantities of interest in our results is of course the inverse
$A$ of this matrix. Upon taking $N$ large, one observes an interesting
property in that a substantial fraction of the diagonal elements of the matrix
$A$ are concentrated at one value. Numerical computation indicates that this
number is close to $14.354...$ independent of $N$. Furthermore, the smallest
of the eigenvalues start off at one value, with many concentrated near it, and
then increase sharply, as illustrated in Figure \ref{Fig EV}. In addition, the 
numerically significant parts of the columns of the matrix $A$ are repetitive.

We denote the smallest eigenvalue by $\lambda_{1}$. Due to the nature of the
integrands taking the form%
\begin{equation}
\int e^{-\lambda_{n}z_{n}^{2}}dz_{n},
\end{equation}
there will be a rapid falloff as $\lambda_{n}$ increases. As a first-order
approximation, one may assume the first $N^{\prime}$ eigenvalues have the
value $\lambda_{1}$ exactly, and ignore the rest as they will not make a
meaningful contribution to the solution. From the minimization perspective, we
are simply neglecting the terms corresponding to points at which the
probability of attaining the minimum is very small. This will lead to a
simplification of the form where one can set $\lambda_{i}=\lambda_{1}$ for 
a number $n'$ of the eigenvalues, and ignore the subsequent ones.

For the shock densities, one observes a similar phenomenon, and expects a
transition probability of the form%
\begin{align}
P_{i,i^{\prime}} &  \tilde{=}\left(  2\pi\right)  ^{-\frac{n}
{2}}\int_{-\infty}^{\infty}\sum_{k=1}^{n_{i}}\sum_{k^{\prime}=1}%
^{n_{i^{\prime}}}\left(  d_{k^{\prime}}-d_{k}\right)  \int_{-\infty}^{s
}...\int_{-\infty}^{s}e^{-\frac{1}{2}\left(  \bar{s}-\tilde{\mu}\right)
^{T}A_{i}\left(  \bar{s}-\tilde{\mu}\right)  }\nonumber\\
&  \cdot e^{-\left(  \frac{\lambda_{l}}{2}\left[  U^{T}\left(  \bar{s}%
-\bar{\mu}\right)  \right]  _{k}\right)  ^{2}-\left(  \frac{\lambda_{l}}%
{2}\left[  U^{T}\left(  \bar{s}-\bar{\mu}\right)  \right]  _{k^{\prime}%
}\right)  ^{2}}ds_{N+1,1}...ds_{1,1}ds.
\end{align}
where $\tilde{\mu}$ and $A$ now correspond to the mean and inverse covariance
matrix on segments $i$ and $i^{\prime}$ together.

In Section 2, the covariance matrix for Brownian bridge in the range where it
is defined is very similar to that of Brownian motion, despite some key
differences in the processes. Outside of $\left[  0,T\right]  $ 
(under linear transformation), the
variance of the integrated Brownian motion is constant. In applying this
initial condition, one must set $g^{\prime}\left(  x\right)  =g\left(
x\right)  =0$ outside of a finite range $\left[  0,T\right]$. In fact, one 
observes the same qualitative behavior
for the eigenvalues of its inverse covariance matrix. This will lead to a
result similar to (\ref{Thm4}) for the solution to the conservation law with
Brownian bridge initial data. Another interesting application is that of a
stationary Ornstein-Uhlenbeck process, i.e. with a covariance matrix invariant
under translation, which may facilitate computations for further results. The
advantage of this process is that it has constant variance, making it
advantageous for some physical models where the randomness is fairly
homogenous, rather than Brownian motion where the process is fixed at $0$ and
variance grows with time.

\begin{figure}
[htp]
\begin{center}
\includegraphics[width=\linewidth]%
{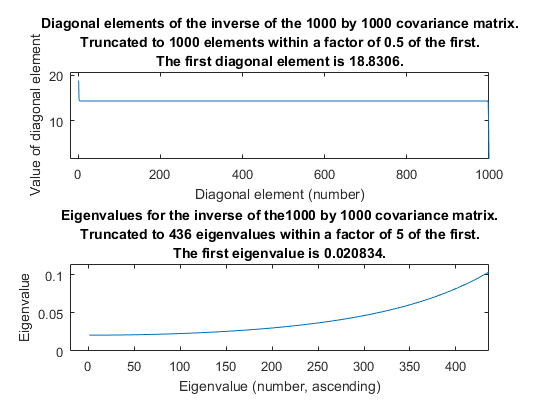}%
\caption{When $N$ is large, the diagonal elements of the inverse covariance
matrix $A$ are concentrated very close to a particular value which can be
calculated as approximately 14.3... The eigenvalues are also concentrated near
one value, and sharply increase. Larger eigenvalues can be neglected due to
the asymptotic behavior of the error function.}%
\label{Fig EV}%
\end{center}
\end{figure}

\subsection{Approximation to smooth flux}

A classical result by Dafermos \cite{D1} utilizing shock techniques describes
how solutions to conservation laws with smooth flux can be approximated by
those of polygonal flux. Our theorems can be combined with Dafermos' work in
order to use polygonal flux as a building block for an arbitrary smooth flux
function with randomness of the form of any Gaussian stochastic process.

One may also take the limits of $N\rightarrow\infty$ and make the partitions
$\left\{  r_{i,j}\right\}  $ increasingly fine in our solution
(\ref{Thm4}). In this case, one should obtain the result in \cite{FM}
for which Burgers' equation is considered with Brownian initial conditions and
exact solutions are computed. Our approach thus provides an extension of
exact, closed-form results without relying on the exact properties of a
particular flux function.

\section{Appendix: results for Gaussian stochastic processes}

Here we present sketches of the proof of several of the results we cited
earlier in Section 2.

\begin{proof}
[Proof of Proposition \ref{Prop 2.3}](b) Define $\left\{  t_{i}\right\}
_{i=1}^{n}$ as before, we need to show that we can write%
\begin{align}
S\left(  t_{1}\right)   &  =a_{11}R_{1}+...+a_{1m}R_{m}+\mu_{1}\nonumber\\
&  ...\nonumber\\
S\left(  t_{n}\right)   &  =a_{n1}R_{1}+...+a_{nn}R_{m}+\mu_{n}%
\end{align}
We can write the integral as an approximate sum with interpolation points
$\left(  r_{1},...,r_{n}\right)  $ given as $r_{i}:=\frac{j}{Nt_{n}}$, so that%
\begin{equation}
S\left(  t\right)  \tilde{=}S_{N}\left(  t\right)  :=\sum_{i=1}^{N_{t}}%
\frac{1}{N}W\left(  \frac{j}{N_{t_{n}}}\right)  \text{, }\frac{N_{t}}{N_{t_{n}}}\leq
t<\frac{N_{t}+1}{Nt_{n}}%
\end{equation}
By definition, then%
\begin{align}
NS_{N}\left(  t_{1}\right)   &  =W\left(  \frac{1}{N}\right)  +W\left(
\frac{2}{N}\right)  +...+W\left(  \frac{N_{t_{1}}}{N}\right) \nonumber\\
&  ...\nonumber\\
NS_{N}\left(  t_{n}\right)   &  =W\left(  \frac{1}{N}\right)  +W\left(
\frac{2}{N}\right)  +...+W\left(  \frac{N_{t_{n}}}{N}\right)  ,
\label{nsumibm}%
\end{align}
i.e. the only difference in the identities in (\ref{nsumibm}) is where one
terminates the sum. Since $W\left(  t\right)  $ is Gaussian stochastic by the
part (a), we may write%
\begin{equation}
W\left(  \frac{i}{N}\right)  =b_{1i}R_{1}+...+b_{1m}R_{m}+\tilde{\mu}_{i}
\label{bmmultivariate}%
\end{equation}
By substituting (\ref{bmmultivariate}) into (\ref{nsumibm}) and defining new
constants%
\[
a_{i1}:=\frac{\left(  b_{i1}+b_{i2}+...+b_{i,Nt_{1}}\right)  }{N},
\]
we obtain the result that integrated Brownian motion is also a Gaussian
stochastic process by using the Central Limit Theorem for functions
(\cite{PK}, p. 399).

We do not prove (c) and (d) here as the ideas are similar. In particular,
since a Brownian bridge is a linear combination of $W\left(  t\right)  $ and
$W\left(  T\right)  $, it can easily be expressed in terms of independent,
identically distributed normal random variables.
\end{proof}

\begin{proof}
[Proof of Theorem \ref{Thm 2.1}](a) We first consider the covariance matrix
for times $s$ and $t$. Since piecewise drift does not alter the covariance
matrix, we assume without loss of generality that this term is zero. Letting
$W\left(  t\right)  $ be a Brownian motion, one observes%
\begin{align}
Cov\left[  \int_{0}^{s}W\left(  r\right)  dr,\int_{0}^{t}W\left(  r\right)
dr\right]   &  =\mathbb{E}\left\{  \int_{0}^{t}W\left(  r\right)  dr\int
_{0}^{t}W\left(  r^{\prime}\right)  dr^{\prime}\right\}  \label{covIBM}\\
&  =\int_{0}^{s}\int_{0}^{t}\mathbb{E}\left\{  W\left(  r\right)  W\left(
r^{\prime}\right)  \right\}  dr^{\prime}dr\nonumber\\
&  =\int_{0}^{s}\int_{0}^{t}\min\left\{  r,r^{\prime}\right\}  drdr^{\prime
}\nonumber\\
&  =s^{2}\left(  \frac{t}{2}-\frac{s}{6}\right)  .
\end{align}
The cross term, $\mathbb{E}\left\{  \int_{0}^{t}W\left(  r\right)  dr\right\}
\mathbb{E}\left\{  \int_{0}^{t}W\left(  r^{\prime}\right)  dr^{\prime
}\right\}  $, vanishes since $\mathbb{E}\left\{  W\left(  t\right)  \right\}
=0$ for Brownian motion.

Using a discretization argument as in \cite{RO}, one may also construct
(\ref{covIBM}) from first principles.

(b) Let $Z\left(  t\right)  $ be an integrated Brownian bridge constructed as
in Definition \ref{Def 2.3}. As in part (a) we consider the sum approximation
to the Brownian bridge assuming no drift, and consider first only two points
$0\leq s\leq t\leq T$. We make note of (\ref{covMatrixIBM}) and also note%
\begin{equation}
Cov\left[  W\left(  T\right)  ,\int_{0}^{t}W\left(  r\right)  dr\right]
=Cov\left[  W\left(  t\right)  ,\int_{0}^{t}W\left(  r\right)  dr\right]
\label{covW1}%
\end{equation}
since $W\left(  T\right)  -W\left(  t\right)  $ is independent of events at
$t$ and before. Continuing from (\ref{covW1}), we have%
\begin{align}
Cov\left[  W\left(  T\right)  ,\int_{0}^{t}W\left(  r\right)  dr\right]   &
=\mathbb{E}\left\{  W\left(  t\right)  \int_{0}^{t}W\left(  r\right)
dr\right\} \nonumber\\
&  =\int_{0}^{t}\mathbb{E}\left\{  W\left(  t\right)  \right\}  W\left(
r\right)  dr\nonumber\\
&  =\int_{0}^{t}\mathbb{E}\left\{  \left(  W\left(  t\right)  -W\left(
r\right)  +W\left(  r\right)  \right)  W\left(  r\right)  \right\}
dr\nonumber\\
&  =\int_{0}^{t}\mathbb{E}\left\{  W\left(  r\right)  ^{2}\right\}
dr=\frac{t^{2}}{2}, \label{covMatrixW1Wr}%
\end{align}
noting that the second term in the definition of covariance vanishes since
$\mathbb{E}\left\{  W\left(  t\right)  \right\}  =0$. One then writes,
expanding the definition of $Z\left(  t\right)  $ and collecting terms from
(\ref{covW1}) and (\ref{covMatrixW1Wr}):%
\begin{equation}
Cov\left[  Z\left(  s\right)  ,Z\left(  t\right)  \right]  =s^{2}\left(
\frac{t}{2}-\frac{s}{6}\right)  -\frac{1}{T}\frac{t^{2}s^{2}}{4}.
\label{covMatrixIBMdiscrete}%
\end{equation}
for times $0\leq\left\{  t_{i}\right\}  _{i=1}^{N}\leq T$.

Now consider $0\leq s\leq T<t.$\ Then we have%
\[
Cov\left[  Z\left(  s\right)  ,Z\left(  t\right)  \right]  =Cov\left[
S\left(  s\right)  -\frac{s}{T}W\left(  T\right)  ,S\left(  T\right)
-W\left(  T\right)  \right]
\]

One obtains similar results by (stochastically) reflecting the process to
consider it on the interval $\left[  -T,T\right]  $ for symmetry. One can also
obtain these results by discretizing the process and writing the approximation
of the integral, then limiting to the continuum using the Central Limit
Theorem for functions. We also note that although the process $Z\left(
s\right)  $ is constant outside the interval $\left[  0,T\right]  $ (or
$\left[  -T,T\right]  $ depending on the definition), $Z\left(  s\right)  $
does not vanish for $s>T$ with probability $1$.
\end{proof}

\section*{Acknowledgments}

The author thanks the referee for review of the manuscript.

\section*{Conflict of interest}
The author declares no conflicts of interest in this paper.


\begin{thebibliography}{99}                                                                                               %

\bibitem {A}D. Applebaum, {\it Levy Processes and Stochastic Calculus}, 
2 Eds, Cambridge: Cambridge University Press, 2009.

\bibitem {BE}J. Bertoin, {\it Levy Processes}, Cambridge: Cambridge University Press, 1996.

\bibitem {BI}P. Billingsley, {\it Probability and Measure}, New York: Wiley, 2012.

\bibitem {BG}Y. Brienier, E. Grenier, {\it Sticky particles and
scalar conservation laws}, SIAM J. Numer. Anal., \textbf{35} (1998), 2317--2328.

\bibitem {CA}C. Caginalp, {\it Minimization Solutions to Conservation
Laws with Non-Smooth and Non-Strictly Convex Flux}, AIMS Mathematics, 
3 (2018), 96-130.

\bibitem {CD}M. Chabanol, J. Duchon, {\it Markovian solutions of
inviscid Burgers equation}, J. Stat. Phys., \textbf{114} (2004), 525--534.

\bibitem {CEL}M. Crandall, L. Evans, P. Lions, {\it Some
properties of viscosity solutions of Hamilton-Jacobi equations}, Trans. Amer.
Math. Soc., \textbf{282} (1984), 487--502.

\bibitem {D1}C. Dafermos, {\it Polygonal approximations of solutions of
the initial value problem for a conservation law}, J. Math. Anal. Appl.,
\textbf{38} (1972), 33--41.

\bibitem {D2}C. Dafermos, {\it Hyberbolic Conservation Laws in Continuum
Physics}, 3 Eds, New York: Springer, 2010.

\bibitem {ERS}E. Weinan, G. Rykov, G. Sinai, {\it Generalized variational
principles, global weak solutions and behavior with random initial data for
systems of conservation laws arising in adhesion particle dynamics}, Commun.
Math. Phys., \textbf{177} (1996), 349--380.

\bibitem {EV}C. Evans, {\it Partial Differential Equations}, 2 Eds,
New York: Springer, 2010.

\bibitem {FM}L. Frachebourg, P. Martin, {\it Exact statistical
properties of the Burgers equation}, J. Fluid Mech., \textbf{417} (2000), 323--349.

\bibitem {GT}D. G. Gilbarg \&\ N. S. Trudinger, {\it Elliptic Partial Differential
Equations of Second Order}, New York: Springer, 1977.

\bibitem {GR}P. Groeneboom, {\it Brownian motion with a parabolic drift
and Airy functions}, Probab. Theory Rel., \textbf{81} (1989), 79--109.

\bibitem {HA}M. Hairer, J. Maas, H. Weber, {\it Approximating rough
stochastic PDEs}, Comm. Pure Appl. Math., \textbf{67} (2013), 776--870.

\bibitem {HO}H. Holden, N. Risebro, {\it Front Tracking for
Hyperbolic Conservation Laws}, New York: Springer, 2015.

\bibitem {HP}E. Hopf, {\it The partial differential equation
$u_{t}+uu_{x}=\mu u_{xx}$}, Comm. Pure Appl. Math., \textbf{3} (1950), 201--230.

\bibitem {KR}D. Kaspar, F. Rezakhanlou, {\it Scalar conservation laws
with monotone pure-jump Markov initial conditions}, Probab. Theory Relat.
Fields, \textbf{165} (2016), 867--899.

\bibitem {KA}A. Kaufman, H. Lim, J. Glimm, {\it Conservative front
tracking: the algorithm, the rationale and the API}, Bulletin of the Institute
of Mathematics, \textbf{11} (2016), 115--130.

\bibitem {KS}D. Khoshnevisan, Z. Shi, {\it Chung's Law for Integrated
Brownian Motion}, T. Am. Math. Soc., \textbf{350} (1998), 4253--4264.

\bibitem {LA1}P. Lax, {\it Hyperbolic systems of conservation laws
II,} Comm. Pure Appl. Math., \textbf{10} (1957), 537--566.

\bibitem {LA2}P. Lax, {\it Hyperbolic systems of conservation laws and
the mathematical theory of shock waves}, Society for Industrial and Applied Mathematics, Conference
Board of the Mathematical Sciences Regional Conference Series in Applied
Mathematics, 1973.

\bibitem {M}G. Menon, {\it Complete integrability of shock clustering and
Burgers turbulence}, Arch. Ration. Mech. An., \textbf{203} (2012), 853--882.

\bibitem {MP}G. Menon, R. Pego, {\it Universality classes in
Burgers turbulence}, Comm. Math. Phys., \textbf{273} (2007), 177--202.

\bibitem {MS}G. Menon, R. Srinivasan, {\it Kinetic theory and Lax
equations for shock clustering and Burgers turbulence}, J. Stat. Phys.,
\textbf{140} (2010), 1195--1223.

\bibitem {PK}M. Pinsky, S. Karlin, {\it An Introduction to
Stochastic Modeling}, Burlington: Elsevier, 2011.

\bibitem {RO}S. Ross, {\it Introduction to Proabbility Models}, 10 Eds,
Burlington: Elsevier, 2010.

\bibitem {RY}H. Royden, P. Fitzpatrick, {\it Real Analysis}, 4 Eds, 
Boston: Prentice Hall, 2010.

\bibitem {SC}Z. Schuss, {\it Theory and Applications of Stochastic
Processes: An Analytical Approach}, New York: Springer, 2010.

\bibitem {SS}S. Shankar, {\it Burgers Equation in 1D\ and 2D}, 2012. Available from: 
\newline
https://www.mathworks.com/matlabcentral/fileexchange/38087-burgers-equation-in-1d-and-2d?focused=5246985\&tab=function

\bibitem {SL1}M. Slemrod,  (2013) {\it Admissibility of the weak solutions for
the compressible Euler equations, n$\geq$\emph{2}}, Philos. T. R. Soc. A, \textbf{371} (2013), pp20120351.

\bibitem {SH}D. She, R. Kaufman, H. Lim, et al.
{\it Handbook of Numerical Analysis}, Elsevier \textbf{17}, 383--402.

\bibitem {VO}A. Vol'pert, {\it Spaces BV and quasilinear equations},
Mat. Sb. (N.S.), \textbf{73} (1967), 255--302.
\end{thebibliography}
\end{document}